\documentclass{amsart}
\usepackage[dvips]{color}
\usepackage{graphicx}
\usepackage{epsfig}
\usepackage{tikz}

\usepackage{latexsym}
\usepackage{amssymb}
\usepackage{amsmath}
\usepackage{amsthm}
\usepackage{amscd}
\theoremstyle{plain}
\newtheorem{thm}{Theorem}[section]
\newtheorem{lem}[thm]{Lemma}
\newtheorem{defin}[thm]{Definition}
\newtheorem{prop}[thm]{Proposition}
\newtheorem{rem}[thm]{Remark}

\newtheorem{cor}[thm]{Corollary}
\newtheorem{notation}[thm]{Notation}
\theoremstyle{definition}

\newtheorem{claim}{Claim}

\newcommand{\calS}{{\mathcal S}}

\newcommand{\T}{{\mathcal T}}

\newcommand{\HH}{{\mathbb H}}
\newcommand{\R}{{\mathbb R}}
\newcommand{\ZZ}{{\mathbb Z}}

\newcommand{\Teich}{Teich\-m\"uller~}
\newcommand{\PMF}{\mathcal{PMF}}
\newcommand{\MF}{\mathcal{MF}}

\DeclareMathOperator{\Ext}{Ext}
\DeclareMathOperator{\twist}{twist}

\title{Limits in $\PMF$ of Teichm\"uller geodesics}
\author{Jon Chaika}\thanks{Research of JC partially supported by DMS 1004372, 1300550.}
\author{Howard Masur}\thanks{Research of HM partially supported by DMS 0905907; DMS1205016}
\author{Michael Wolf}\thanks{Research of MW partially supported by DMS 1007383 and the Morningside Center (Tsinghua Univ.). HM and MW appreciate the support of the GEAR Network (DMS 1107452, 1107263, 1107367).}
\begin{document}

\maketitle

\centerline{\today}

\begin{abstract} We consider the limit set in Thurston's compactification $\PMF$ of Teichm\"{u}ller space of some Teichm\"{u}ller geodesics defined by quadratic  differentials with minimal but not uniquely ergodic vertical foliations.  
We show that 
a) there are quadratic differentials so that the limit set of the geodesic 
is a unique point,
b) there are quadratic  differentials so that the limit set is a line segment,
c) there are quadratic  differentials so that the vertical foliation is ergodic and there is a line segment as limit set, and
d) there are quadratic  differentials so that the vertical foliation is  ergodic and there is a unique point as its limit set.
These give examples of divergent Teichm\"{u}ller geodesics whose limit sets overlap and Teichm\"{u}ller geodesics that 
stay a bounded distance apart but whose limit sets are not equal. A byproduct  of our methods is a construction of a Teichm\"{u}ller geodesic and a simple closed curve $\gamma$ so that the hyperbolic length of the geodesic in the homotopy class of $\gamma$ varies between increasing and decreasing on an unbounded sequence of time intervals along the geodesic.
\end{abstract}

\section{Introduction}
Let $S=S_{g,n}$ be a surface of genus $g$ with $n$ punctures. We assume $3g-3+n\geq 1$.  
Let $\T(S)$ be the Teichm\"uller space of $S$ with the Teichm\"uller metric, and let $X\in \T(S)$ a Riemann surface.     We denote by $q=\phi(z)dz^2$ a meromorphic quadratic differential with at most simple poles at the punctures of $X$.  Let $\|q\|=\int_X |\phi(z)|dz^2|$ its area. Each   $q$ such that $\|q\|=1$  determines a Teichm\"uller geodesic  ray $X_t; 0\leq t<\infty$.  Namely, for each $t$ take the Teichm\"uller map $f_t:X\to X_t$ which expands along the horizontal trajectories of $q$ by  $e^t$ and contracts along the vertical trajectories by $e^t$.   Consequently the unit sphere of qudratic differentials on $X$ can be thought of as a visual  boundary of $\T(S)$ as seen from $X$, and gives a compactification of $\T(S)$, depending on the base point $X$.

Teichm\"uller  space via the uniformization theorem is also the space of hyperbolic structures. Using hyperbolic geometry, Thurston constructed  a compactification 
$\PMF$  of Teichm\"uller space. (See \cite{F})  This compactification is more natural in that it is basepoint free.  In particular, the action of the mapping class group on $\T(S)$ extends  naturally to $\PMF$ and this is the basis for Thurston's classification of elements of the mapping class group.  An interesting question is to compare these two compactifications.   Ultimately this question involves relating complex structures with hyperbolic structures via the uniformization theorem. 

 It was already known by work of Kerckhoff \cite{Ke} that the two compactifications are not the same. 
He showed that the mapping class group acting on $\T(S)$ which has a continuous  extension to $\PMF$ does not have a continuous extension to the visual sphere. 
For $X_t$ a geodesic ray  in $\T(S)$, an interesting question then is to find the limit points of the ray $X_t$ in $\PMF$.  What makes this problem at times complicated  is that Teichm\"uller geodesics are determined by varying the flat metric of the quadratic differential while the Thurston compactification is defined in terms of the conformally equivalent hyperbolic metric and so the difficulty is in comparing these metrics.   In the case that the vertical foliation $(F,\mu)$ of $q$ is Strebel, which means that the vertical trajectories are closed, Masur \cite{masur boundary} showed that the limit exists and is the {\em barycenter} of the simplex of invariant measures for the vertical foliation.  Namely one assigns equal weights to each closed curve. If the vertical foliation is  uniquely ergodic, he also showed that there is a unique limit which is the projective class $[F,\mu]$ itself. Lenzhen \cite{Le} found the first example where there is not a unique accumulation point for a ray. In her example, the surface consisted of two tori glued along a slit and the slit was in the vertical direction. The flow in each torus is minimal but because the slit is in the vertical direction,  the vertical foliation is not minimal. 

The question then arises of investigating quadratic differentials whose vertical foliation is minimal but not uniquely ergodic.  In that case, there is a simplex of invariant measures whose extreme points are ergodic measures.
C. Leininger, A. Lenzhen, and K. Rafi \cite{LLR} recently found such examples of $1$-parameter families on the five-times punctured sphere such that for any quadratic differential whose vertical foliation is in this family, the limit point of the geodesic in $\PMF$ is the entire one-dimensional simplex.   Their construction is topological in nature and their measured foliations are found by taking limits in $\PMF$ of curves under  Dehn twists.      

 In this paper we consider a class of examples of geodesics exhibiting different and in some sense complementary phenonema from the Leininger-Lenzhen-Rafi examples. These new examples first arose in work of Veech \cite{vskew1} who constructed non-uniquely ergodic minimal dynamical systems in the context of skew products over rotations by a number $\alpha$.   They may be interpreted in terms of translation surfaces  $(X_c,\omega_c)$ of genus $2$, consisting of a pair of tori glued along a horizontal slit. The tori are rectangular  tori to which we have applied a shear or parabolic transformation defined in terms of $\alpha$. In a second description, we have two square tori glued along a slit and we have a flow in a direction with angle $\alpha$.  The slit length is determined by $\alpha$.  These  are described in Figures 1 and 2 below. In the first description, the first return map of the flow in the vertical direction  to a horizontal interval  is the given interval exchange. 
 
The non-unique ergodicity of the system means that the vertical flow fits into  a $1$-parameter family of (oriented) measured foliations denoted by $(F,\mu_{c})$, 
where $F$ is the fixed topological foliation and $\mu_c$ is the transverse measure and $-1\leq c\leq 1$.  Each transverse  measure $\mu_c$ is of the form $\frac{1}{2}(1-c)\mu_-+\frac{1}{2}(1+c)\mu_+$ where $\mu_-= \mu_{-1}$ and $\mu_+=\mu_1$ are ergodic measures. Our notational convention also distinguishes a non-ergodic but {\it symmetric} flow-invariant transverse measure $\mu_0$. 
Thus we have a family of examples parametrized by $\alpha$ and for each of these there is a $1$-parameter family of invariant measures for the corresponding flow.  


Let $\Lambda\subset \PMF$ be the corresponding $1$-dimensional simplex consisting of this family $[F,\mu_\tau]$ of projective measured foliations.  
(We change the parameter from $c$ to $\tau$ since we will be taking the Teichm\"uller geodesic corresponding to one foliation determined by $\omega_c$ and seeing which limit points, described as foliations $\mu_{\tau}$, are achieved).  
Now each $(X_c,\omega_c)$ determines a Teichmuller geodesic $g_t(X_c,\omega_c)$. 
The main results of this paper are
 \begin{thm}\label{thm:barycenter}
For any of the Veech examples  (any $\alpha$) 
 \begin{enumerate}
 \item  If $c=0$,  there is a unique limit point of $g_t(X_0,\omega_0)$  which is the barycenter $[F,\mu_{0}]\in\Lambda$.
 \item If $c \notin \{-1,1\}$ then the ergodic endpoints $[F,\mu_1]$ and $[F,\mu_{-1}]$ are not in the limit set of $g_t(X_c,\omega_c)$. 
 \end{enumerate}
 \end{thm}
 \begin{thm}\label{thm:ex} If $c \neq\{-1,0 ,1\}$ there are examples of $\alpha$ where the limit set contains the barycenter and other points as well.
 \end{thm}

 In the case where the transverse measure is ergodic different phenomena can occur. 
 \begin{thm}\label{thm:ergodic}
 \begin{enumerate}
  \item There are examples of $\alpha$ for which  the limit set of $g_t(X_1,\omega_1)$ contains the  interval 
from $[F,\mu_0]$ to $[F,\mu_1]$; namely, it contains the barycenter and the corresponding ergodic endpoint   of the interval.
 \item There is an example of $\alpha$ for which  $g_t(X_1,\omega_1)$ converges to $[F,\mu_1]$.  

 \end{enumerate}
 \end{thm}

\begin{rem}By \cite{LM} the geodesic rays in Theorems \ref{thm:barycenter} and \ref{thm:ergodic} diverge from each other in the Teichm\"{u}ller metric  and yet they can share  the barycenter as a limit point in $\PMF$.  By \cite{ivanov}  the ray  in Theorems \ref{thm:barycenter} corresponding to $c=0$ and in \ref{thm:ex}  corresponding to $c\neq 0$  stay a 
bounded distance apart in the Teichm\"{u}ller metric  
 and yet have different limit sets in $\PMF$. 
These two examples highlight ways in which Thurston's compactification behaves differently than the visual boundary of Teichm\"{u}ller space in contrast to the case when the vertical foliation is uniquely ergodic. 
\end{rem}

\begin{rem}
After the introduction of the Veech construction, in Theorem~\ref{thm:CF}   we will be more specific  in how the examples are built.  The behavior of geodesics will depend on the continued fraction expansion of $\alpha$.  
\end{rem}

\begin{rem}
We do not  know the exact limit set  in the first  part of Theorem~\ref{thm:ergodic} and in Theorem~\ref{thm:ex}.   Our Theorems are complementary to those in \cite{LLR} in that they achieve  the entire simplex in $\PMF$ as a limit set while we do not. In fact the most intricate example we give is (2) in Theorem~\ref{thm:ergodic} where we find a geodesic with a unique limit.   
\end{rem}

\begin{rem} By our methods we show that there exists a pair of a simple closed curves and  Teichm\"{u}ller geodesic so that the hyperbolic lengths of these  curves change  from increasing to decreasing an arbitrarily large amount arbitrarily far out along the geodesic. This will be shown by Corollary \ref{cor:decrease}.
This does not contradict the Theorem in \cite{LR} which says that lengths are quasiconvex along geodesics.

\end{rem}

The arguments for the theorems follow a  pattern.  The Teichm\"uller geodesics we construct, when projected from the Teichm\"uller space to the Moduli space, will at all times   represent approximately a pair of  tori glued along a slit that is short in the flat metric.  
For each of the geodesics we study, we anticipate this decomposition by defining curves on the surface which will eventually become the slit, and the  basis curves for each of the tori which will become moderate or even short along the geodesic.  
Our definitions of these curves will determine their flat geometry, and their relative intersection numbers. Most of our arguments then become a study of the hyperbolic geometry of the resulting tori, especially the hyperbolic lengths of the curves defining the original pair of tori, by a study of the lengths of the slits, and the lengths of the basis curves of the tori.  

In terms of that description of the structure of the argument, we organize the paper as follows. In Section~$2$, we recall Veech's original construction of a non-uniquely ergodic interval exchange transformation and adapt it for our purpose of creating particular Teichm\"uller geodesics with the desired asymptotics. 
In Section~$3$, we recall the characterizations of convergence in $\PMF$ that we need, and we define the slits and the basis curves for each of the tori complementary to the slit;  we record their basic flat geometric and topological invariants. 
Section~$4$ describes the geometry and measure theory of the tori that we encounter along the Teichm\"uller geodesic, and this provides enough background to prove the first of our basic results in Section~$5$. In preparation for the proofs of the remaining results, we collect in Section~$6$ some facts on the hyperbolic geometry of these  tori which are glued along a slit.  In particular, we prove two results which allow us to estimate the hyperbolic lengths of some 'test curves' in terms of their intersection numbers with, and the lengths of the slits, and the  basis curves. All of this applied in the final Section~$7$, where we prove the remaining main results by carefully estimating the hyperbolic lengths of the test curves: this involves estimating all of the terms in the formulae we found in Section~$6$.

\begin{notation}
\emph{Rate Comparison.} Given two functions $f(x),g(x)$ we say $f\sim g$ if $\lim_{x\to\infty}\frac{f(x)}{g(x)}\to 1$.
We say $f\asymp g$ if there exists $C$ such that
$$\frac{1}{C}\leq \frac{f(x)}{g(x)}\leq C$$ as $x\to\infty$.

\end{notation}

\begin{notation}
\emph{Geometric Invariants.} Given a hyperbolic metric $\rho$, and a closed geodesic  or geodesic arc $\gamma$, then $\ell_\rho(\gamma)$ is its hyperbolic length. 

Given a quadratic differential $q$ and a time $t$ along   the Teichm\"uller geodesic defined by $q$, and simple closed curve $\gamma$,  denote by $|\gamma|_t$ the length of $\gamma$ at time $t$;  also, $v_t(\gamma)$ will denote the vertical component and $h_t(\gamma)$ the horizontal component of $\gamma$. Note of course that a geodesic in the flat metric is not necessarily a geodesic in the hyperbolic metric so, in discussing the length of a curve $\gamma$, we will always be referring to the geodesic representative of $[\gamma]$.
\end{notation}

\section{The Veech construction}
Let  $<<x>>:=x-\lfloor x \rfloor$ denote the fractional part of $x$.  
Let $0<\alpha<1$ be a real number with continued fraction expansion $$\alpha=[a_1,a_2\ldots a_n\ldots ].$$ 
Let $p_n/q_n$ be the corresponding convergents.  (We recall the recursive relationships $p_n = a_np_{n-1} + p_{n-2}$ and  $q_n = a_nq_{n-1} + q_{n-2}$.)
For the construction of Veech we assume 
there is a  subsequence $n_k$ such that 
\begin{equation}
\label{eq:condA} \tag{A}
\sum_{k=1}^{\infty} (a_{n_k+1})^{-1}<\infty.
\end{equation} 

We further assume  in  Theorem~\ref{thm:ex} 
that the  following additional pair of conditions hold.
\begin{equation}
\label{eq:condB} \tag{B}
a_{n_k}\to\infty.
\end{equation}

\begin{equation}
\label{eq:condC} \tag{C}
q_{n_{k-1}}\log a_{n_k+1}=o(q_{n_k})
\end{equation}
\begin{rem}
If $\alpha$ has dense orbit under the Gauss map then there exists a sequence of $a_{n_i}$ that have these three properties. 
\end{rem}

\begin{rem}
Condition B will control the geometry of surfaces along the geodesic which will allow us to calculate hyperbolic lengths. Condition C will control length of slits.  
\end{rem}
\subsection{IET construction}

 Let $$b=\sum_{k=1}^{\infty} 2<<q_{n_k}\alpha>>$$ 
 and 
$$J=[0,b].$$

Then  (\cite{vskew1})
the interval exchange transformation (IET) $T:[0,1) \times \mathbb{Z}_2 \to [0,1) \times \mathbb{Z}_2$  given by
$$ T(x,i)=(<<x+\alpha>>, i +\chi_{J}(x))$$ is minimal but not uniquely ergodic. In fact it has two ergodic measures $\mu_-,\mu_+$ and Lebesgue measure $\lambda$ is $\frac 1 2 (\mu_-+\mu_+).$ The subintervals are:
\begin{enumerate} 
\item[] $I_1=J \times \{0\},$ 
\item[] $I_2= [b,1-\alpha) \times \{0\},$
\item[] $I_3=[1-\alpha,1) \times \{0\},$
\item[] $I_4=J \times \{1\},$
\item[] $ I_5=[b,1-\alpha) \times \{1\},$
\item[]$ I_6=[1-\alpha, 1) \times \{1\}$
\end{enumerate} and the permutation is $(342615),$ meaning that the third interval goes to the first position, the fourth to the second position and so forth. 
We normalize so that $|I_1|+|I_2|+|I_3|=|I_4|+|I_5|+|I_6|=1$.

Veech proved the following result  in \cite{vskew1}. 
\begin{thm} \label{thm:Veech}
Assume (\ref{eq:condA}) holds. Then there is a positive  Lebesgue measure set $E$ of points such that the orbit starting at $x\in E$ spends asymptotically more than $1/2$ their time in the first three intervals  and a positive measure set of points that asymptotically spend more than $1/2$ their time in the second three intervals. 
\end{thm}

In particular this implies that the ergodic measures assign different lengths to the two rectangles. 
Again following Veech \cite[Section 1]{viet} we can obtain a {1-parameter} family $T_c$ of minimal but not uniquely ergodic IETs.   The length of  the  $j^{th}$ interval $_cI_j$ of the map $T_c$ is $$|_cI_j|=\frac{1}{2}(1-c)\mu_-(I_j)+\frac{1}{2}(1+c)\mu_+(I_j).$$ Moreover these IETs are conjugate to each other.  Recalling that $\mu_c=\frac{1}{2}(1-c)\mu_- +\frac{1}{2}(1+c)\mu_+$,  let $f_c:\cup I_j\to \cup I_j$ given by 
$$f_c(x)=\mu_c([0,x]).$$ It conjugates $T$ to $T_c$.  That is, $$f_c\circ T\circ f_c^{-1}=T_c.$$
For $c\neq 0$, we have  $|_cI_2|\neq |_cI_5|$ and $|_cI_3|\neq |_cI_6|$ but $|_cI_1|=|_cI_4|$.

These IETs arise from first return map of a flow on two tori glued together along a horizontal slit.    Consider the following picture, where the lengths of the intervals are as described by $T$  and the intervals labeled $I_j$ on the top and bottom are identified.

\begin{figure}[h]
\begin{tikzpicture}
\path[-] (0,0) edge (0,-2)
node{$\circ$};
\path[-] (0,0) edge (.4,0)
node[right=.1 cm, above=.1 cm]{$I_1$};
\path[-]  (.5,0) edge (3, 0)
node{*}
node[right=1.5 cm, above=.1cm]{$I_2$};
\path[-] (3.05,0) edge (4.1,0)
node[right=.5cm, above=.1cm]{$I_3$};
\path[-] (0,-2) edge (1.05,-2)
node[right=.5cm, above=.03 cm]{$I_3$};
\path[-] (1.1,-2) edge (1.5,-2)
node{$\circ$}
node[right=.1cm, above=.03cm]{$I_4$};
\path (1.6,-2) edge (4.1,-2)
node{*}
node[right=1.5 cm, above=.03 cm]{$I_2$};
\path (4.4,0) edge (4.8,0)
node{$\circ$}
node[right=.1 cm, above=.1cm]{$I_4$};
\path (4.85,0) edge (7.35,0)
node{*}
node[right =1.5 cm, above=.1 cm]{$I_5$};
\path (7.4,0) edge (8.4,0)
node[right=.5cm, above=.1cm]{$I_6$};
\path (4.35,-2) edge (5.35,-2)
node[right=.5cm,above=.03cm]{$I_6$};
\path (5.4,-2) edge (5.8,-2)
node{$\circ$}
node[right=.1 cm, above=.03cm]{$I_1$};
\path (5.85,-2) edge (8.4,-2)
node{*}
node[right=1.5cm, above=.03 cm]{$I_5$};
\path (8.4,-2) edge (8.4,0);
\path[-] (4.1,-2) edge (4.1,0);
\path[-] (4.35,-2) edge (4.35,0);
\end{tikzpicture}
\caption{}
\end{figure}

\noindent We may also label intervals $J_1,\ldots, J_6$ in that order along the bottom edges where $|J_i|=|I_i|$. 
We set these intervals $J_i$ exactly below the corresponding intervals $I_i$ (with the same subscript) on the top edge.  The first return to the bottom of the (upwards) vertical flow gives the map $T$ on the intervals $J_i$ described above.  For example, a point in the first interval $J_1$ from $[0,b) \times \{0\}$ on the bottom edge travels upward to the top, landing in the interval  $I_1$ which is then identified with a point in the bottom of the second torus in the segment labelled $I_1$ -- this will be in the fifth position. This is in the interval $[\alpha,\alpha+b)\times \{1\}$. 
The interval $J_2$ -- occupying the interval labelled $I_4$ on the bottom edge -- will flow to the top to points in $I_2$ and return to the third position on the bottom. 
The points which are marked $\circ$ are identified with each other and the points that are asterisks are identified with each other.  Each  becomes a singular point of  the translation
surface. 
We can see that the glued surface consists of two closed  tori each cut along  segments $I_1,I_4$ (which are identified in each  torus) and then cross-identifying these cut segments in the different tori. The result is a genus $2$ surface. The union of the glued $I_1,I_4$ is a separating curve also called a {\em slit}.
We can think of these tori in the following way. Each torus is the  square lattice  multiplied by 
$\begin{pmatrix} 1 &-\alpha\\
0 &1
\end{pmatrix}$.   Now, the vertical flow on the original square torus $\R^2/ \ZZ \oplus i\ZZ$ has closed leaves, but the effect of this shearing matrix is to change the vertical flow but preserve the horizontal flow.  It is easy to check that the first return map of the vertical flow on $J_2\cup J_3$ on the first torus is rotation by $\alpha$. The same is true for  $J_5\cup J_6$ on the second torus.
 
Now in the nonsymmetric case of $c\neq 0$ one still has two rectangular tori glued along the slits and where the vertical flow is the given one.  The horizontal lengths however have changed and so therefore has the transverse measure to the vertical flow. 
The area of the tori are no longer equal. The translation surface is denoted by $(X_c,\omega_c)$. 
We denote the corresponding measured foliation by $(F,\mu_c)$.  The underlying leaves of $F$ are the flow lines and the transverse measure is $\mu_c$. 

 Let $\gamma_1,\gamma_2$ be the horizontal closed curves in the two tori, 
and denote by $i((F,\mu_c),\gamma_i)$ the intersection number of  $\gamma_i$ with $(F,\mu_c)$. This is the horizontal length of $\gamma_i$.
We conclude that 
  
\begin{prop}
\label{prop:crit} Suppose $[F,\mu_c]\in\Lambda$.  Then $c$ is determined by $\frac {i((F,\mu_c),\gamma_1)}{i((F,\mu_c),\gamma_2)}$.  In particular $c=0$  if and only if this quotient is $1$.  
\end{prop}

 Our main Theorems now have the following more precise statements. 


\begin{thm}
\label{thm:CF}
Assume  $\alpha$ satisfies (\ref{eq:condA}).
Then the only limit point of $g_t(X_{0},\omega_{0})$ is the barycenter $[F,\mu_{0}]$ itself.  If $c\notin\{-1,1\}$ then the ergodic endpoints are not in the limit set.   
\end{thm}
\begin{thm}
\label{thm:CF2}
If $\alpha$ satisfies  (\ref{eq:condA}), (\ref{eq:condB}), and (\ref{eq:condC}) then the accumulation set  of $g_t(X_c,\omega_c)$ for $c\notin\{-1, 0,1\}$ is a nondegenerate  interval in $\Lambda$ that contains the barycenter. 
\end{thm}

For the next two statements, we recall that the conditions (\ref{eq:condA}), (\ref{eq:condB}), and (\ref{eq:condC}) refer to a given sequence of indices $\{n_k\} \subset \ZZ$; part of the results involve a choice of such a sequence.

\begin{thm}
\label{thm:CF3}
If $$\alpha=[1,1,a_3,a_4,4,4,a_7,a_8,...,(k+1)^2,(k+1)^2,a_{4k+3},a_{4k+4},...],$$ then setting $n_k=4k-3, k\geq 1$,  for some 
 $a_{4k-1}, a_{4k}$,  the accumulation set of $g_t(X_{-1},\omega_{-1})$ of the ergodic foliation is an interval that contains $[F,\mu_-]$ and the  barycenter $[F,\mu_{0}]$. 
\end{thm}
\begin{thm}
\label{thm:CF4}
If $$\alpha=[1,2,4,\ldots ,2^k,\ldots]$$ and $n_k=k$,  then the accumulation set of  $g_t(X_{-1},\omega_-)$  is just the endpoint $[F,\mu_-]$ itself.  
 \end{thm}

\section{Convergence in $\PMF$}
For each Riemann surface $X$ let $\rho_X$ be the hyperbolic metric on $X$. 


\begin{defin}
A projective class of measured foliation $[G,\nu]$ is a limit point of a sequence of metrics $\rho_{t_i}$ if  for any pair of curves $\alpha,\beta$ we have $$\frac{\ell_{\rho_{t_i}}(\alpha)}{\ell_{\rho_{t_i}}(\beta)}\to \frac{i(G,\nu),\alpha)}{i((G,\nu),\beta)}.$$

\end{defin}


We have the following result  which is a slight generalization of the Main  Theorem  in \cite{masur boundary}.
\begin{lem}
\label{lem:continuity}  Let $(X,q)$ be a quadratic differential with vertical foliation $(F,\mu)$.  Assume $(F,\mu)$ is minimal and not uniquely ergodic.  Then any projective limit point $[G,\nu]$ of the sequence of hyperbolic metrics $\rho_t$ corresponding to $g_t(X,q)$ satisfies  $i((F,\mu),(G,\nu))=0$; that is, $G$ is topologically the same as $F$.   
As a corollary, in the case at hand any limit point of $g_tr_{\theta}(X_c,\omega_c)\in \Lambda$. 
\end{lem} 

\begin{proof}
Let $[G,\nu]\in \PMF $ be any  limit of the hyperbolic metrics $\rho_n$ of $g_{t_n}(X,q)$ for some subsequence  $t_n\to\infty$.   Let $(G,\nu)$ again denote  a representative. There exists a sequence $r_n\to 0$ such that $$r_n\rho_n\to (G,\nu)$$ in $\R^{\calS}$. Fix a pants decomposition $P=\{\kappa_i\}$ of the surface  such that 
$i((G,\nu),\kappa_i)>0$ for all $\kappa_i$.  According to (\cite{F} page 140), in a neighborhood $V_\epsilon$  consisting of metrics $\rho$ for which $\rho(\kappa_i)>\epsilon$ there is a projection map $\pi:V_\epsilon\to \MF$ such that for each simple closed curve $\gamma$, there is a constant $C$ such that for all $\rho\in V_\epsilon$, we have   
\begin{equation}
\label{eq:double}
i(\pi(\rho),\gamma)\leq \ell_\rho(\gamma)\leq i(\pi(\rho),\gamma)+C.
\end{equation}
Now  $r_n\ell_{\rho_n}(\gamma)\to i((G,\nu),\gamma)$ for all $\gamma$.
From (\ref{eq:double}) for each $\gamma$, the fact that $r_n\to 0$ and $C$ is fixed, we   conclude   
$$r_ni(\pi(\rho_n),\gamma)\to i((G,\nu),\gamma)$$ as well so that 
$$r_n\pi(\rho_n)\to (G,\nu).$$

On the other hand there is a sequence  $\beta_n$ of curves that become moderate length  along the sequence.  That is $$\ell_{\rho_n}(\beta_n)=O(1).$$
(In the case at hand it is either the simple closed curve corresponding to a slit or a curve $\sigma_k$).   In the flat metric on the image surface under the Teichm\"uller map, these curves have bounded length  which implies that, on the base surface $q$, they are long and the horizontal component of their holonomy goes to $0$ as $n\to\infty$. That is, $$i((F,\mu),\beta_n)\to 0.$$  It follows that, after passing to subsequences, projectively $\beta_n$ converges to a foliation $[H,\nu']$ topologically equivalent to $(F,\mu)$; that is, there exists $s_n\to 0$ such that $s_n\beta_n\to (H,\nu')$ and $i((H,\nu'),(F,\mu))=0$. 

But now applying the left side of (\ref{eq:double}) we have 
$$\lim_{n\to\infty}i(r_n\pi(\rho_n),s_n\beta_n)\leq r_ns_n\ell_{\rho_n}(\beta_n)\to 0.$$ 
Since $r_n\pi(\rho_n)\to (G,\nu)$ and $s_n\beta_n\to (H,\nu')$, by continuity of intersection number, we conclude that $$i((G,\nu),(H,\nu'))=0,$$ and this implies $(G,\nu)$ topologically equivalent to $(H,\nu')$;  hence to $(F,\mu)$ as well. 

\end{proof}

The following lemma is the criterion we will use  to show we get ergodic endpoints in $\PMF$. 
\begin{lem}
\label{lem:ergodiclimit}
Suppose we are given a sequence of metrics $\rho_n$ converging to a point in $\Lambda$ 
and a  sequence of closed curves  $\sigma_n$ converging in $\PMF$ to the ergodic measure $\mu_-$. Then  $|\frac{\ell_{\rho_n}(\gamma_1)}{\ell_{\rho_n}(\gamma_2)}-\frac{i(\sigma_n,\gamma_1)}{i(\sigma_n,\gamma_2)}|\to 0$ if, and only if, we also have
$\rho_n\to [F,\mu_-]$. 
\end{lem}
\begin{proof}
The Lemma follows from Proposition~\ref{prop:crit} and the fact that  $\frac{i(\sigma_n,\gamma_1)}{i(\sigma_n,\gamma_2)}\to \frac{i((F,\mu_-),\gamma_1)}{i((F,\mu_-),\gamma_2)}$. 

\end{proof}

\subsection{Flows on $(X_c, \omega_c)$}

Let $(X,\omega)$ be a general translation surface with distance $d(\cdot,\cdot)$ and 
let $\phi^t_\theta$ the directional flow on $(X,\omega)$ in direction $\theta$ at time $t$. 
The following is Euclidean geometry. 
\begin{lem}
\label{lem:Euclidean} Suppose $x$ is a point on $(X,\omega)$ and  $\theta,\psi$ are a pair of directions.  Then $d(\phi^t_{\theta}(x), \phi^{t\sec(|\theta-\psi|}_{\psi}(x))= t\tan(|\psi-\theta|)$ as  long as no trajectory leaving $x$ in an angle between $\theta$ and $\psi$ hits a singularity within time $t$. 
\end{lem}

Before continuing we need some basic facts about rotations. 
Let $T:[0,1)\to [0,1)$ be given by $$T(x)=x+\alpha, \mod 1.$$
\begin{lem}\label{basic dio} $\frac 1 {(a_{n+1}+2)q_n}<|T^{q_n}x-x)|<  \frac 1 {q_{n+1}}<\frac 1 {a_{n+1}q_n}$.
\end{lem}

\begin{proof}
See, for example,  \cite{khinchin}, four lines before equation 34, which says
\begin{equation}
\label{eq:K}
\frac 1 {q_k(q_k+q_{k+1})}<|\alpha-\frac{p_k}{q_k}|<\frac 1 {q_kq_{k+1}}.
\end{equation} By multiplying everything by $q_k$ and recalling that $q_{k+1}=a_{k+1}q_k+q_{k-1}$ and so $q_{k+1}\geq a_{k+1}q_k$ and $q_{k}+q_{k+1}\leq (a_{k+1}+2)q_k$ we obtain the lemma.
\end{proof}

\begin{lem}\label{hit next}
If $I$ is a half open interval such that $|I|\leq \frac{1}{3}<<q_i\alpha>> $  and $T^jx\in I$ then $T^{j+k}x\notin I$ for $|k|<q_{i+1}$.
\end{lem}
\begin{proof}  First we note that if $y,z \in I$ then $|y-z|< |I|\leq \frac{1}{3q_{i+1}}$, the last inequality by (\ref{eq:K}).  Now if  $T^j(x)\in I$, then 
 $$|T^jx-T^{j+k}x)|=|x-T^{k}x|>|x-T^{q_i}x|>\frac{1}{3q_{i+1}}$$ for all $ |k|<q_{i+1}$ with $|k|\notin \{0,q_i\}$. The last inequality follows from  the left hand inequality in Lemma~\ref{basic dio} and the fact that $q_{i+1}=a_{i+1}q_i+q_{i-1}$. 
 
Thus, $|T^jx-T^{j+k}x)|$ exceeds the diameter of $I$, and so both $T^jx$ and $T^{j+k}x$ cannot both be contained in $I$.
\end{proof}

Now we return to the current setting of the vertical flow on the  translation surface $(X_{0},\omega_{0})$. 
It will be convenient to 
describe the flow and translation surface somewhat differently. 
Recall that rotation by $\alpha$ arises by the first return to the horizontal circle under the flow in direction $\text{arctan}(\alpha):=\theta$ (measured from the vertical) on the square (unit) torus. Take two square tori, and glue them together by a pair of horizontal slits of length $b$. The flow in direction $\theta$ on this flat surface $Y$ of genus two gives a second way of viewing the surface.  This topological flow then admits a $1$-parameter family of invariant transverse measures.
  When we change the measures, the components of curves in direction $\theta$ stay the same but the components in directions perpendicular to $\theta$ change.

\begin{figure}[h]
\begin{tikzpicture}
\path[->](0,0) edge (.7,1.2);
\draw (0,.7) arc (90:50:.5);
\path[-] (0,0) edge (2,0) edge (0,2);
\path[-] (1.1,2)
node{$\circ$};
\path (1.4,1.95) node{*};
\path (1.4,-.05) node{*};
\path (1.1,0) node{$\circ$};
\path (1.2,2.3) node{$I_1$};
\path (1.2, .3) node{$I_4$};
\path (.2,.8) node{$\theta$};
\path[-] (2,0) edge (2,2);
\path[-] (2,2) edge (0,2);

\path[-] (3,0) edge (5,0) edge (3,2);
\path (4.1,2) node{$\circ$};
\path (4.4,1.95) node{*};
\path[-] (5,0) edge (5,2);
\path[-] (5,2) edge (3,2);
\path (4.1,0) node{$\circ$};
\path (4.4,-.05) node{*};
\path (4.3,2.3) node{$I_4$};
\path (4.3, .3) node{$I_1$};
\end{tikzpicture}
\caption{}
\end{figure}


In the picture above,  the two pieces from the circle to the asterisk labeled $I_1$ are identified. We do a similar identification for the two pieces labeled $I_4$. Otherwise, within each square, opposite sides are identified. Following the resulting identifications, the circle and asterisk have cone angle $4\pi$. One can consider the first return map to the union of the two bottoms. The resulting IET has permutation $(342615)$. It has third and sixth intervals with length $\tan(\theta)$ comprising the right segments of the horizontal sides of each torus. The first and fourth intervals are segments of length $b$ on the left hand side of the two squares, which flow into the slits. The second and fifth intervals are what is left over.


Now the vertical flow $(F,\mu_c)$ with transverse measure $\mu_c$ by itself does not determine a translation structure. One needs a horizontal flow. One way to specify this in the current situation is to fix the genus two surface $Y$ above.  By, for example, a theorem of Hubbard-Masur \cite{HM}, there is on $Y$ a translation surface structure $(Y,\omega_c)$ whose vertical measured foliation is $(F,\mu_c)$.  
We note the difference with the figure and the description from the previous section. There we had a family of translation surfaces $(X_c,\omega_c)$ on varying Riemann surfaces. Here there is a fixed Riemann surface $Y$ and 
translation surface $(Y,\omega_c)$ on $Y$.  Our main theorem applies equally well to $(Y,\omega_c)$.

We next connect the previous Lemmas~\ref{lem:Euclidean} and \ref{basic dio} by finding closed curves which approximate the flow direction and whose slopes are defined in terms of the convergents of the continued fraction of $\alpha$.

\begin{lem} 
\label{lem:periodic}
Recall the continued expansion $\alpha=[a_1,a_2,\ldots ]$ and that $\tan(\theta)=\alpha$.    There is a constant $C$ with the following property. Let  $p_{n_k}/q_{n_k}$ be the fraction whose continued fraction expansion is $[a_1,...,a_{n_k}]$.   Let $\theta_k$ be the angle with the vertical axis so that $\tan\theta_k=p_{n_k}/q_{n_k}$.  Then for a set of $p$  of measure more than  $\frac 1 2$ of points $p$ we have  $$d(\phi_{\theta_k}^{\sec(|\theta_k-\theta|) t}(p),\phi_{\theta}^t(p))<\frac{C}{q_{n_k}a_{n_k+1}},$$ 
for all $t< q_{n_k}$.
\end{lem}


\begin{proof}  By \eqref{eq:K} and the right hand inequality in Lemma \ref{basic dio}, 
$$|\tan(\theta_k)-\tan(\theta)|<\frac1 {a_{n_k+1}q_{n_k}^2}.$$ 
Now the rotation $R_{\tan(\theta_k)}$ is periodic with period   $q_{n_k}$.   So the above inequality gives that $$d_{S^1}(R^n_{\tan(\theta_k)}(x),R^n_{\tan(\theta)}(x))<\frac 1 {a_{n_k+1}q_{n_k}}$$ for all $n<q_{n_k}$. We wish to apply Lemma~\ref{lem:Euclidean}.  Fix $N\geq 2$. To guarantee the hypotheses 
it is enough to know that the flowline $\phi^t_{\theta}$ is not in a $\frac{N}{a_{n_{k}+1}q_{n_k}}$ neighborhood of the singularities  for $t<q_{n_k}$.  Because the flow is measure preserving this is true for all but $$q_{n_k}\frac{2N}{a_{n_k+1}q_{n_k}}=\frac{2N}{a_{n_k+1}}<1/2,$$ of the measure of $\omega$ for $k$ large enough.
\end{proof}

Since, by Theorem~\ref{thm:Veech}, there are flow lines that spend unequal amounts of time in each torus, we have the following.
Define   $\sigma_j^k; j=-,+$ to be the closed orbits in direction $\theta_k$. 
\begin{cor}
\label{cor:unequal} The closed orbits $\sigma_j^k$ spend  unequal amounts of time in the two tori. 
We may also consider these closed curves $\sigma_j^k$ as closed curves on the original translation surface as given in Figure 1.  On that surface they are almost vertical.
\end{cor}


\subsection{Slits and torus decompositions}

Suppose $(X,\omega)$ is a translation surface of genus $2$ with a pair of zeroes.  The hyperelliptic involution $\tau$ interchanges the zeroes and takes $(X,\omega)$ to $(X,-\omega)$. (\cite{FK} page 98)

\begin{defin}
A slit is a saddle connection $\sigma$ joining the pair of zeroes 
 such that $\sigma\cup\tau(\sigma)$ is a separating curve that separates the surface into a pair of tori.  The involution fixes each complementary torus.
\end{defin}

\begin{lem} 
For $c=0$, vectors of the form $(b+2p,2q)$ where $(p,q)$ is an integer vector are slits.
\end{lem}
\begin{proof} The proof is in \cite{masur-tabachnikov}. We sketch the proof here.  The fact that $2p,2q$  are even means that the segment $w'$ joining $(0,0)$ to $(b+2p,2q)$ on each torus, when projected to the torus, not only has the same endpoints as the segment $w$ joining $(0,0)$ to $(b,0)$ but intersects it an odd number of times in its interior.  Thus they divide each other into an even number of segments, say $2n$.  There are 
$n$ parallelograms $P_{i,j}; j=1,\ldots, n$ bounded by segments of $w$ and $w'$ on the torus $T_i$ with the following propert: one can enter a parallelogram by crossing any of the $2n$ segments.   We can then form two new tori $T_-',T_+'$ separated by the union of the pair of slits $w'$ as  follows.  The torus $T_-'$ is made up of $(T_-\setminus \cup_{j=1}^n P_{1,j})\cup_{j=1}^n P_{2,j}$ and similarly for $T_+'$. 
\end{proof}

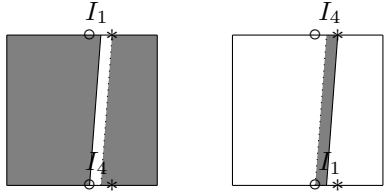
\begin{figure}[h]\caption{The slit is the union of the black line and its involution, the dotted line. $T_-'$ is the shaded region}
\begin{tikzpicture}

\path [fill=gray] 
 (0,0)--(1.1,0)--(1.25,2)--(0,2);
 \path [fill=gray] (1.25,0)--(2,0)--(2,2)--(1.4,2);
 \path [fill=gray] (4.25,0)--(4.4,2)--(4.25,2)--(4.1,0);
\path[-] (0,0) edge (2,0) edge (0,2);
\path[-] (1.1,2)
node{$\circ$};
\path (1.4,1.95) node{*};
\path (1.4,-.05) node{*};
\path (1.1,0) node{$\circ$};
\path (1.1,0) edge (1.25,2);
\path[dotted] (1.25,0) edge (1.4,2);
\path (1.2,2.3) node{$I_1$};
\path (1.2, .3) node{$I_4$};
\path[-] (2,0) edge (2,2);
\path[-] (2,2) edge (0,2);

\path[-] (3,0) edge (5,0) edge (3,2);
\path (4.1,2) node{$\circ$};
\path (4.4,1.95) node{*};
\path[-] (5,0) edge (5,2);
\path[-] (5,2) edge (3,2);
\path (4.1,0) node{$\circ$};
\path[dotted] (4.1,0) edge (4.25,2);
\path (4.25,0) edge (4.4,2);
\path (4.4,-.05) node{*};
\path (4.3,2.3) node{$I_4$};
\path (4.3, .3) node{$I_1$};
\end{tikzpicture}
\end{figure}
\label{Fig:picture1}

We now inductively define a sequence of slits $\zeta^k$ that approximate the minimal foliation $(F,\omega_{0})$  and which will become short for some times under the Teichmuller flow $g_t(X_{0},\omega_{0})$.  They will be disjoint from the $\sigma_j^k$.  Here we return to Figure 1.  
We define the interval $I_1=J\times \{0\}$ to be the first slit $\zeta^1$.   
We have $$|\zeta^1|=|I_1|=\sum_{k=1}^{\infty} 2<<q_{n_k}\alpha>>.$$ 
By the above formula and Lemma \ref{basic dio}, the length $|I_1|$ can be estimated in terms of the coefficients $a_{i}$.  By choosing them sufficiently large  sufficiently large,  
we can assume $$|I_1|<\  \frac{1}{3}<<q_{n_1-1}\alpha>>.$$

Starting at $0$, and flowing in the vertical direction, by Lemma~\ref{hit next} we do not return to $I_1$ until we have returned $q_{n_1}$ times to $I_2\cup I_3\cup I_5 \cup I_6$,  and then we return at  distance $<<q_{n_1}\alpha>>$ from $0$. Then we return again to $I_1$ after another $q_{n_1}$ returns  at distance $2<<q_{n_1}\alpha>>$ from $0$. We now define  the second slit $\zeta^2$ to be the homotopy class of the arc which consists of  the flow line starting at $0$ and ending at $x_2=2<<q_{n_1}\alpha>>$ followed by the segment $J_1\subset I_1$ from $x_2$  to $(b,0)$.  This latter segment has length $$|J_1|=\sum_{k\geq 2} 2<<q_{n_k}\alpha>>.$$
We can assume $$|J_1|< q_{n_1}\alpha\ \text{mod}\ 1.$$    By the above statements $\zeta^2$  has no self-intersections. 
The length $|\zeta^2|$ of $\zeta^2$ is estimated by $$|\zeta^2|\asymp\sum_{k\geq 2} 2<<q_{n_k}\alpha>>+2q_{n_1}.$$
The horizontal component  $$h(\zeta^2)=O(\sum_{k\geq 2}2<<q_{n_k}\alpha>>).$$

Now we start at $x_2$, the  left endpoint of $J_1$, and apply the return map $q_{n_2}$ times arriving at a point in $I_1$.   By Lemma~\ref{hit next} we do not hit $J_1$ before we have accomplished $q_{n_2}$ such returns.  We do this an additional $q_{n_2}$ times  arriving  at a point $x_3$.  Thus we have built the third slit $\zeta^3$  by following the flow line from $0$  a total of $2q_{n_1}+2q_{n_2}$ return times to  $x_3$, and then take the remaining part of $I_1$ from $x_3$ to $(b,0)$.  
Then $$h(\zeta^3)=O(\sum_{k\geq 3}<<2q_{n_k}\alpha>>).$$  We repeat this inductively to build to  build the   $k^{th}$ slit $\zeta^k$ for all $k$. Then $$h(\zeta^k)=O(\sum_{j\geq k}<<2q_{n_j}\alpha>>)=O(\frac{1}{a_{n_k+1}q_{n_k}}).$$
Then for sufficiently large but fixed $N$ depending only on the implied multiplicative bound above, since the  $\sigma^k_i$ do not come within  $\frac{N}{a_{n_k+1}q_{n_k}}$ of the singularities, they miss the slit.
\begin{defin} \label{defn:tori}
This  pair of slits $\zeta^k$ divide the surface into a pair of tori, which we will denote by $T_-^k,T_+^k$.
\end{defin}
In light of the discussion above, the tori $T_-^k,T_+^k$ contain the curves $\sigma_i^k$.  
They are interchanged by the map that exchanges the two tori.

\begin{defin} \label{defn:beta}
We refer to the convergent previous to $\sigma_j^k$ as $\beta_j^k$, also contained in $T_i^k$. 
\end{defin}
The curves $\beta_j^k$ correspond to the vectors $(p_{n_k-1},q_{n_k-1})$.  
 We summarize this discussion with the following lemmas. 

As an explanation of notation, the curves $\sigma_{\pm}^k,\beta_{\pm}^k$
depend on the changing torus decompositions $T_{\pm}^k$ while the curves $\gamma_1,\gamma_2$ whose hyperbolic lengths we want to measure on changing hyperbolic surfaces are fixed.

\begin{lem}
\label{lem:slitsize}
\begin{enumerate}
\item The horizontal  component of the curve $\sigma_i^k$ of slope $p_{n_k}/q_{n_k}$  is 
 $h(\sigma_i^k)\asymp \frac{1}{q_{n_k}a_{n_k+1}}$.
\item   $h(\zeta^k) =\sum_{j\geq k} 2<<q_{n_j},\alpha>>=O(\frac{1}{q_{n_k}a_{n_k+1}})$
\item $v(\zeta^k)=O(2q_{n_{k-1}})$.
\end{enumerate}
\end{lem}
Now let $\alpha_-,\alpha_+$ be any pair of curves that are images of each other by the map that exchanges the two tori.  This includes the curves $\gamma_i$ whose intersections with the $[F,\mu_c]$ determines $c$.

\begin{lem}
\label{lem:crossings}  We have
\begin{enumerate}
\item  $i(\alpha_-,\sigma_+^k)=i(\alpha_+,\sigma_-^k)$
\item $i(\alpha_-,\sigma_-^k)=i(\alpha_+,\sigma_+^k)$
\item $i(\gamma_i,\sigma_j^k)\asymp q_{n_k}$.
\item For each $i=1,2$,  $\lim_{k\to\infty} \frac{i(\gamma_i,\sigma_-^k)}{i(\gamma_i,\sigma_+^k)}=\frac{\mu_-(\gamma_i)}{\mu_+(\gamma_i)}\neq 1$. 
\item $i(\zeta^k,\gamma_i)\asymp q_{n_{k-1}}$.
\end{enumerate}
\end{lem}

We similarly have, for the previous convergents $\beta_j^k$, the estimates

\begin{lem}
\label{lem:previous}
\begin{enumerate}
\item $i(\gamma_i, \beta_i^k) \asymp q_{n_k}/a_{n_k}$.
\item  $i(\gamma_1,\beta_+^k)=i(\gamma_2,\beta_-^k)$
\item $i(\gamma_1,\beta_-^k)=i(\gamma_2,\beta_+^k)$
\item for $i=1,2$, $\lim_{k\to\infty} \frac{i(\gamma_i,\beta_-^k)}{i(\gamma_i,\beta_+^k)}=\frac{\mu_-(\gamma_i)}{\mu_+(\gamma_i)}\neq 1$. 
\end{enumerate}
\end{lem}




\section{Teichm\"uller flow}

Let
 $$ g_t = \begin{pmatrix} e^{t} & 0 \\ 0 & e^{-t} \end{pmatrix}, \qquad
r_\theta = \begin{pmatrix} \cos \theta & \sin \theta \\ - \sin \theta
  & \cos \theta \end{pmatrix}.$$

We wish to know which curves and slits get short under the flow $g_t$ in the direction of the nonuniquely ergodic foliations $(F,\mu_c)$.
That is we consider $g_t(Y,\omega_c)$. We will need careful estimates of the relative sizes of intersection numbers and the vertical and horizontal components of the important curves (slits and convergents $\sigma_i^k, \beta_i^k$ in the three regimes of the symmetric case $c=0$, the non-ergodic but asymmetric case $0<|c|<1$, and in the ergodic cases $c = \pm1$.

\subsection{The symmetric case.} These estimates are easiest to describe in the symmetric case $c=0$ since 
in that case the tori are square tori with Lebesgue measure.

When we consider the second way of viewing the flow as the vertical flow on the rotated square tori,  the vertical flow is fixed, but since the underlying Riemann surface is different, the horizontal foliation varies so that vertical lengths change by a small multiplicative factor (and similarly for horizontal lengths).  From the point of view of our computations of lengths, this will not matter, for as we will see in the proofs of  Theorems~\ref{thm:CF}, \ref{thm:CF2} and \ref{thm:CF3},  the computations we need are always ones that hold up to fixed multiplicative error. 

After the rotation $r_\theta$, the direction of the flow is vertical.  Then any slit or simple closed curve determines a holonomy vector.  We are interested in its horizontal and vertical components.
The horizontal component expands by a factor of $e^t$ under $g_t$, and the vertical component contracts by the same factor. 
We begin with the following.

\begin{lem}\label{lem:close to square}Let $\theta_k$ be a periodic direction for the square torus with $(\cos(\theta),\sin(\theta))=\frac{1}{\sqrt{p_{n_k}^2+q_{n_k}^2}}{(p_{n_k},q_{n_k}})$ with  $[a_1,...,a_{n_k}]$ the continued fraction expansion of $p_{n_k}/q_{n_k}$ which corresponds to the curve $\sigma^k$.  Define $t_k$ so that  $e^{t_k}=\sqrt{p_{n_k}^2+q_{n_k}^2}$.  Then 
$$d(g_{t_k}r_{\theta_k}T,T)=O(\frac 1 {a_{n_k}}).$$ 
\end{lem}
\begin{proof}
In the sequence $\{p_{n_k}/q_{n_k}\}$ of best approximates to $\alpha$, let $p'/q'$ be the element of the sequence prior to $p_{n_k}/q_{n_k}$ (so that 
$\frac{p'}{q'} = \frac{p_{{n_k}-1}}{q_{{n_k}-1}}$).
Then $p_{n_k}q'-q_{n_k}p'=1$ and 
$$\max \{p'/p_{n_k},q'/q_{n_k}\} <\frac{1}{a_{n_k}}.$$  The fraction $p'/q'$ corresponds to a curve $\sigma'$. Consider the matrix $M_k=g_{t_k}r_{\theta_k}$ as an element of  $SL_2(\mathbb{R})/SL_2(\mathbb{Z})$.  We apply it to the curves $\sigma^k$ defined by  vector  $(q_{n_k},-p_{n_k})$ and $\sigma'$ defined by  the vector $(q',-p')$. 
The result is
\begin{multline*}
M_k \begin{pmatrix} q'&q_{n_k}\\-p'&p_{n_k}\end{pmatrix}=\\
\begin{pmatrix} \sqrt{p_{n_k}^2+q_{n_k}^2} &0 \\
0& \frac 1 {\sqrt{p_{n_k}^2+q_{n_k}^2}}
\end{pmatrix}
\begin{pmatrix} \frac{p_{n_k}}{\sqrt{p_{n_k}^2+q_{n_k}^2}} & \frac{q_{n_k}}{\sqrt{p_{n_k}^2+q_{n_k}^2}}\\
\frac{-q_{n_k}}{\sqrt{p_{n_k}^2+q_{n_k}^2}} & \frac{p_{n_k}}{\sqrt{p_{n_k}^2+q_{n_k}^2}}
\end{pmatrix}
\begin{pmatrix} q'&q_{n_k}\\
-p'&-p_{n_k}
\end{pmatrix}=\\
\begin{pmatrix} 1 &0\\
\frac{-p_{n_k}p'-q_{n_k}q'}{p_{n_k}^2+q_{n_k}^2}& 1
\end{pmatrix}
.
\end{multline*}

\end{proof}

Using the above we wish to describe the geometry of $(Y,\omega_c)$ after flowing time $t_k$ where  
$e^{t_k}=\sqrt{p_{n_k}^2+q_{n_k}^2}$.
\begin{prop}
 \label{prop:geom}
With $e^{t_k}=\sqrt{p_{n_k}^2+q_{n_k}^2}$ and $\theta$ the direction of the minimal foliation $(F,\mu_{0})$, the surface
$g_{t_k}r_{\theta} (Y,\omega_{0})$ consists of two identical tori $T_-^k,T_+^k$ glued along the $k^{th}$ slit such that 
\begin{itemize} 
\item The distance between each $T_j^k$ and a standard square torus is $O(\frac 1 {a_{n_k}}+\frac{1}{a_{n_k+1}}).$
The almost vertical curves are $\sigma_i^k; i=-,+$.  The almost horizontal curves are $\beta_i^k$. 
\item The  $k^{th}$ slit $\zeta^k$ has  vertical component $$v(\zeta^k)=O(\frac{q_{n_{k-1}}}{q_{n_k}}+ \frac{1}{a_{n_k+1}})$$ and  horizontal component $$h(\zeta^k)\asymp  \frac{1}{a_{n_k+1}}.$$
\item
 $i(\gamma_i,\zeta^k)=\sum _{j\leq k-1} 2q_{n_j}=O(2q_{n_{k-1}})$.
\end{itemize}

\end{prop}
\begin{proof}
The first bullet comes directly from the Lemma~\ref{lem:close to square}, the first statement of Lemma~\ref{lem:slitsize}, and the definition $e^{t_k}=\sqrt{p_{n_k}^2+q_{n_k}^2}$.
The second  bullet  is a consequence of the second and third conclusions of  Lemma~\ref{lem:slitsize}. 
The third bullet is from the construction of the slits. 
\end{proof}

As an application of these estimates, we find some geometric limits which will be important for us later.

\begin{lem}
For  every $a>0$,  there is a sequence of times $t_n$ such that  $g_{t_n}r_{\theta}(Y,\omega_{0})$  converges to  a pair of $(a,\frac 1 {a})$ rectangular punctured tori  in the Hausdorff topology on flat surfaces. 
\end{lem}

\begin{proof}
  By (\ref{eq:condA}) and (\ref{eq:condB})   $$\frac{1}{a_{n_k}}+\frac{1}{a_{n_k+1}}\to 0.$$ It follows from the first conclusion of Proposition~\ref{prop:geom} that  for all $t$, for $k$ large enough depending on $t$, the surface
 $g_{t_k+t}r_{\theta}(Y,\omega_{0})$ consists of two tori that are  the image of $T_j^k$ that are still close to rectangular.  The second  conclusion says that  for $k$ large enough compared to $t$, the $k^{th}$ slit is still short. Thus for any  $a$  we choose $t=\log a$ so that 
$g_{t_k+t}r_{\theta}(Y,\omega_{0})$ converges to the  $(a, \frac{1}{a})$ punctured torus. 
\end{proof}

\subsection{The non-ergodic asymmetric case.}

We use the symmetric case to understand the more general case when the measure is asymmetric but not ergodic, i.e. $c \in (-1, 1), c\ne 0$.

\begin{prop}
\label{prop:geom2} In the nonsymmetric case and nonergodic case (i.e. $c\notin\{-1,0,1\}$), the second and third bullets of Proposition~\ref{prop:geom} hold, but the tori $T_-^k,T_+^k$ are not identical. The curves $\sigma_i^k$ and $\beta_i^k$ have the same vertical components as in the symmetric case,  but the horizontal components differ.   The areas of the tori $T_-^k,T_+^k$  are asymptotically different as $k\to\infty$.  
The  ratio of the areas  converges to $\frac{1+c}{1-c}$.      
The ratio of horizontal lengths converges to $\frac{1+c}{1-c}$.
\end{prop}
\begin{proof}
 
Let us first focus on the symmetric case. Consider the $k^{th}$ separating slit $\zeta^k$ given by Lemma \ref{lem:slitsize}. As we noted in Definition~\ref{defn:tori}, the slit divides the surface into two symmetric tori denoted $T_-^k$ and $T_+^k$. As $k$ goes to infinity, each of the two ergodic measures for the vertical flow are increasingly supported in one of the tori. That is, let $\mu_-,\mu_+$ be the ergodic measures; we can think of each as a two dimensional measure defined by transverse measure times length along the orbit. Then $\mu_-(T_-^k)$ converges to either $1$ or $0$ and   $\mu_-(T_+^k)$ converges to either $1$ or $0$.  The same is true of $\mu_+$.  See for example \cite[Section 3.1]{masur-tabachnikov}. For convenience let's assume the tori are compatibly named (so $\{(\mu_-(T_-^k),\mu_-(T_+^k))\}_{k=1}^{\infty}$ has a unique limit point, $(1,0)$).
Changing weights commutes with this decomposition, so on $(X_c,\omega_c)$ we have corresponding tori $(T_-^k),(T_+^k)$ whose Lebesgue measure is $\frac{1}{2}(1-c)\mu_-(T_-^k)+\frac{1}{2}(1+c)\mu_+(T_-^k)$ and $\frac{1}{2}(1-c)\mu_-(T_+^k)+\frac{1}{2}(1+c)\mu_+(T_+^k)$.  Then this sequence of decompositions of $(X_c,\omega_c)$ have Lebesgue measure going to $\frac{1}{2}(1-c)$ and $\frac{1}{2}(1+c)$.

\end{proof}

\subsection{The ergodic case.}

We now consider the ergodic case $g_tr_\theta (Y,\omega_1)$ with ergodic measure $\mu_+$. 
The vertical components are the same as in the nonergodic case, but now we  see that the ratio of the areas of the tori goes to $0$.  
In the next lemma we can regard  $\mu_{\pm}$ as an area measure given by the transverse measure times the length along the flow direction.

\begin{lem} 
\label{lem:erg}
Assume $\mu_+(T_+^k)\geq \mu_+(T_-^k)$. 
  Then    $\frac{1}{4a_{n_k+1}} \leq \mu_+(T_-^k)\leq \sum_{j\geq k}^{\infty} \frac {2}{a_{n_j+1}}$.
\end{lem}
\begin{proof}  For the lower bound, notice that the area of each torus is at least the horizontal component of the slit multiplied by the minimal return time of the vertical flow to the slit. At time $q_{n_k}$ the horizontal component is at least $\frac{1}{2a_{n_k+1}}$ and the torus has an almost vertical side of length comparable to $1$, and this bounds from below the length of a minimal return time.  
       
We next prove the upper bound. We will use the notation that $\lambda$ is Lebesgue measure. 
   The measures $\mu_i, i=-,+$ are  supported on the disjoint flow invariant sets $$A_i^\infty=\liminf T_i^k=\{x:\exists N: x\in \cap_{k=N}^\infty T_i^k\}.$$ Thus 
   \begin{multline}
\label{eq:Delta}
\mu_+(T_-^k)=\mu_+(T_-^k\cap A_2^\infty)=\lambda(T_-^k\cap \cup_{n=k+1}^\infty \cap_{j=n}^{\infty} T_+^j)\leq\\
\lambda(T_-^k\cap T_+^{k+1})+\sum_{j\geq k+1} \lambda (T_-^k\cap (T_+^{j+1}\setminus T_+^j)).
\end{multline}  The set   $T_+^j\Delta T_+^{j+1}$ is a union of parallelograms bounded by the two slits whose  area is bounded by the cross product of the pair of vectors.   Thus by (2) and (3) of Lemma~\ref{lem:slitsize},  for each $j$, 
we have 
\begin{equation*}\lambda(T_+^j\Delta T_+^{j+1})\leq \frac{2}{a_{n_j+1}}.
\end{equation*}
Applying this estimate to both terms in (\ref{eq:Delta}) gives the Lemma.
\end{proof}

\begin{cor} In the notation of Definition~\ref{defn:beta}, we have

\label{cor:smallhor}$$\lim_{j\to\infty}\frac{\mu_-(\beta_-^j)}{\mu_+(\beta_-^j)}=\infty.$$
\end{cor}
\begin{proof}
We have that the two dimensional measure $\mu_{\pm}(T_-^k)$ is, up to bounded  multiplicative error, given by $$v(\sigma_-^k)h_{\pm}(\beta_-^k)=v(\sigma_-^k)\mu_{\pm}(\beta_-^k),$$ where again in the last term we are using transverse measure. 
Now by 
Lemma~\ref{lem:erg} we have $$\frac{\mu_-(T_-^k)}{\mu_+(T_-^k)}\to\infty,$$ and so the Corollary follows. 
\end{proof}

We now give the description of $T_-^k$ and $T_+^k$ in the ergodic case. At time $t_k$, by the argument for the first conclusion of Proposition~\ref{prop:geom},  the surface $g_{t_k}r_\theta(Y,\omega)$ is within $O(\frac{1}{a_{n_k}}+\frac{1}{a_{n_k+1}})$  of two glued  tori.  Each has  a side of flat length comparable to one which is almost vertical.  Let $\sigma_i^k$ denote the almost-vertical sides and $\beta_i^k$ the curves that come from the penultimate convergent:  see Definition~\ref{defn:tori}.  On $T_+^k$ the curve $\beta_+^k$ is almost horizontal.  On the other hand, Lemma~\ref{lem:erg} states that the areas of the two tori are very different.  This means that $\beta_-^k$ has much shorter horizontal length.  In fact we conclude that for some constant $C$
\begin{equation} \label{tinytorus}
\frac{C}{a_{n_k+1}}\leq |\beta_-^k|_{t_k}\leq \sum_{j=k}^{\infty}\frac {2}{a_{n_j+1}}.
\end{equation}

We will apply this inequality in the proof of Theorems~\ref{thm:CF3} and \ref{thm:CF4}, when we treat the situation of the accumulation set of the flow of the ergodic measures.

\section{Proof of Theorem \ref{thm:CF}}
The proofs in this section hold more generally in the presence of symmetries of the flat structure that interchange ergodic measures.

{\it Proof of Theorem \ref{thm:CF}:} We first consider the case of the limits of $g_tr_\theta(Y,\omega_0)$. Assume we are given $(Y,\omega_{0})$.
Lemma~\ref{lem:continuity} says that any limit point  of $g_tr_\theta(Y,\omega_c)$ lies in $\Lambda$.  
Proposition~\ref{prop:crit} says that any such  limit  $[F,\omega_\tau]$ is determined by the ratio of the limit of the hyperbolic lengths of the curves $\gamma_1\gamma_2$. Now because $c=0$ the tori $T_-^k,T_+^k$ are identical.  The map that interchanges the tori also interchanges $\gamma_1$ and $\gamma_2$.  Thus any time $t$ it follows that $$\ell_{\rho_t}(\gamma_1)=\ell_{\rho_t}(\gamma_2),$$ 
so that any limit $[F,\mu_\tau]\in\Lambda$ satisfies $$i((F,\mu_\tau),\gamma_1)=i((F,\mu_\tau),\gamma_2).$$  But then 
$\tau=0$, concluding the argument in this case that the only limit point is  the barycenter. 

We next consider the case where $c \notin \{-1,1\}$, aiming to show that $\mu_-$ and $\mu_+$ are not in the accumulation set of $g_t(Y,\omega_c)$. We begin with

\begin{lem}For each $c \notin\{-1,1\}$ there exists $a>0$ so that if $\alpha_-,\alpha_+$ are any pair of curves on the base surface $(X_{0},\omega_{0})$ that are image of each others under the map that exchanges the two tori then  
 $$\underset{t \to \infty}{\limsup}\, \frac{l_t(\alpha_-)}{l_t(\alpha_+)}<a.$$
\end{lem}

\begin{proof} 
At any time $t$ there  is an piecewise affine  map of the surface to itself which exchanges the two tori, exchanging the pair of saddle 
connections
that form  the slit,  and in the flat coordinates $(x,y)$ is of the form $$A(x,y)=(A_1(x,y),y).$$ This is possible since vertical lengths are the same on each torus. The map interchanges $\alpha_-$ and $\alpha_+$.  Now since horizontal lengths are comparable,  in fact the map $A$ has  uniformly bounded dilatation.
Such maps change hyperbolic lengths by a bounded factor.  
\end{proof}

 Continuing the proof of Theorem~\ref{thm:CF}, let $a>0$ be the constant from the above lemma.  Now  Corollary~\ref{cor:smallhor} says that  $\frac{\mu_-(\beta_-^j)}{\mu_+(\beta_-^j)}\to\infty$.  This implies that for  $j$ large enough   
$$\underset{k \to \infty}{\lim}\, \frac{i(\beta_-^j,\sigma_-^k)}{i(\beta_-^j,\sigma_+^k)}>a.$$ 

Now by Lemma \ref{lem:crossings}(1-2), because $\beta_+^j$ and $\beta_-^j$ are exchanged (Corollary~\ref{cor:unequal}) by the map which exchanges the tori, we have
  $i(\beta_+^j,\sigma_-^k)=i(\beta_-^j,\sigma_+^k)$. Thus for the fixed pair of curves $\beta_-^j,\beta_+^j$ we have 
\begin{align*}
\lim_{k\to\infty}\frac{i(\beta_-^j,\sigma_-^k)}{i(\beta_+^j,\sigma_-^k)} &=\lim_{k\to\infty}\frac{i(\beta_-^j,\sigma_-^k)}{i(\beta_-^j,\sigma_+^k)} \frac{i(\beta_-^j,\sigma_+^k)}{i(\beta_+^j,\sigma_-^k)}\\
&=\lim_{k\to\infty}\frac{i(\beta_-^j,\sigma_-^k)}{i(\beta_-^j,\sigma_+^k)}\\
&>a\\
&> \underset{t \to \infty}{\limsup}  \frac{l_t(\beta_-^j)}{l_t(\beta_+^j)}.
\end{align*}
  It follows from Lemma~\ref{lem:ergodiclimit} that $\mu_{-}$ is not a limit point.
Similarly $\mu_{+}$ is not a limit point. 


\section{Convergence of hyperbolic lengths} \label{sec:hyperbolic}
Because we are interested in controlling the limits of \Teich geodesics in the Thurston compactification, we need to be able to estimate not just the flat lengths of curves on our translation surfaces (with estimates like those in the last section), but we also need to control the corresponding hyperbolic lengths. In this section, we record some basic tools that will allow us to interpret the estimates in the previous section into estimates for hyperbolic lengths.

We begin with a comment about how simple closed curves meet collars. Suppose $(X_k,\omega_k)$ is a sequence of translation surfaces that arise from tori $T_-^k,T_+^k$ glued along a slit of length $\epsilon_k\to 0$ based at the origin of each. Suppose for each $j=-,+$ that $$\lim_{k\to\infty} T_j^k=
S_j^0,$$ which is either a punctured  torus, with its puncture at the origin, or a thrice-punctured sphere with one puncture at the origin.   Any sufficiently small neighborhoods $U_-,U_+$ of the origin in each $S_j^0$ can be considered to be neighborhoods of the corresponding short geodesic for large $k$.  The next lemma is well-known to the community; we include a proof to facilitate the exposition.

\begin{lem} \label{thm:non-crossing curves shallow}
There exist $U_-,U_+$ such that for large $k$, any simple closed hyperbolic geodesic on $X_k$ that enters $U_j$ also crosses the short geodesic.  
\end{lem}
\begin{proof} 

We begin with a preliminary lemma.  Choose  a small number $\epsilon$ and consider a 
hyperbolic surface $S$ with a short geodesic $\alpha$ of length $\epsilon$. Let $\mathcal{C} \subset S$ be an embedded collar around $\alpha$, and let $V=V_{\epsilon}$ be one component of $\mathcal{C} \sim \alpha$; thus $V$ may be parametrized as $V = \{c_- < \Im z < \frac{\pi}{2\epsilon} | 0 < \Re z <1\}$ with sides identified by the the translation $z \mapsto z+1$. The hyperbolic metric on V is represented in these coordinates by $ds^2 = \epsilon^2 \csc^2( \epsilon y) |dz|^2$, and the core geodesic $\alpha$ is represented at height $y=\frac{\pi}{2\epsilon}$.   Here we denote the length of the long boundary curve $\{y=c_-\} = \Gamma_- \subset \partial V$ by $\ell_-$.  We choose $c_-$ and $\ell_->0$ so that for any $\epsilon$ sufficiently small, there are always embedded collars bounded by constant geodesic curvature curves of length $\ell_-$.

We claim

\begin{lem} \label{lem:ShallowNoncross}
Fix $\ell_-$.  Then there is a number $d$ independent of the short length $\epsilon$ so that if $W_d$ denotes the $d$-neighborhood in $V$ of $\Gamma_- \subset \partial V$, and if $\gamma \subset S$ is any simple closed geodesic on $S$ which meets $\Gamma_-$, then either $\gamma \cap (W_d \cup V) \subset W_d$ or $\gamma$ crosses the core geodesic $\alpha$.
\end{lem}

The point here is that if $\gamma$ does not cross the core geodesic, then $\gamma$ invades the collar to only a depth $d$.

\begin{proof}
Let $\gamma$ be a a simple geodesic arc in $V$ that does not cross the core geodesic but meets $\Gamma_-$.  Up to translating our model, we may assume that $\gamma$ enters $V$ at the point $(0, c_-)$, the bottom left corner of the parametrizing strip. 
Either $\gamma$ exits $V$ before meeting the right boundary $\{x=1\}$ of the strip, or $\gamma$ crosses the strip and meets the right boundary at $(1, y_1)$, where $y_1 > c_-$. 
(This geodesic cannot exit the strip along the left boundary, as the segment and the left boundary would bound a disk-like domain in hyperbolic space with two geodesic arcs, and it cannot exit the strip through the core geodesic $\{y = \frac{\epsilon}{2\pi}\}$, by hypothesis.) 
There are a compact family of segments of the first type parametrized by the points $(x, c_-) \in \Gamma_-$ where the segment exits $V$. A computation in planar hyperbolic geometry shows that such arcs only penetrate $V$ up to at most a distance $d < \infty$ from $\partial V$, with $d$ independent of $\epsilon$.

On the other hand, we claim that this geodesic $\gamma$ cannot meet the right edge of the strip, either. For, if the initial segment, say $\gamma_-$, of $\gamma$ does meet the right boundary of $V$, say at $(1, y_1)$, then since $(1, y_1)$ is equivalent to $(0, y_1)$ in our model, we can represent $\gamma$ as continuing its path moving in a rightward path from $(0, y_1)$.  
Now, this new segment, say $\gamma_+$, of $\gamma$ cannot cross the initial segment $\gamma_-$ of $\gamma$, as we have assumed that the geodesic $\gamma$ has no self-intersections; thus, as in the case of the first segment $\gamma_-$, the segment $\gamma_+$ must also reach and cross the right edge of the strip, say at $(1, y_2)$.  
Of course, since we have noted that $\gamma_+$ must lie above $\gamma_-$ to avoid self-intersections, we see that for the coordinates $y_i$ of the right endpoints of $\gamma_i$, we have $y_2 > y_1$.  As before, the point $(1, y_2)$ is equivalent to the point $(0, y_2)$, and we continue in this way obtaining an increasing sequence $\{y_n\}$ of endpoints of segments $\gamma_n$ of $\gamma$.
Next, we have already noted that the geodesic $\gamma$ does not cross the core geodesic (located at height $y=\frac{\pi}{2\epsilon_n}$), so since $\gamma$ (and hence $\gamma \cap V$) is also compact, the sequence $\{y_n\}$ realizes a maximum height, say $y_m < \frac{\pi}{2\epsilon}$.
Focusing our attention on the $m+1$st segment $\gamma_{m+1}$ which begins at $(0, y_m)$, we see that, as in the case of the segment $\gamma_-$, this segment $\gamma_{m+1}$ must also either end on $\partial V$ at some point $(x, c_-)$, or on the right edge at $(1, y_{m+1})$: by assumption that $y_m$ is at the maximum height, we have that $y_{m+1} < y_m$.
Yet this latter alternative is not possible as the segment $\gamma_m$ divides the strip into two components, one containing the initial point $(0,y_m)$ of $\gamma_{m+1}$ and one containing the terminal point $(1,y_{m+1})$ of $\gamma_{m+1}$, here using on the left edge that $y_m > y_{m-1}$ and on the right edge that $y_m > y_{m+1}$. Similarly, the segment cannot exit through $\partial V$ at some point $(x, c_-)$, as this is also separated from $(0,y_m)$ by $\gamma_m$.

Thus since $W_d$ is constructed to contain all the points at distance $d$ from the longer component $\Gamma_- \subset\partial V$, then any simple closed curve which enters $W_d$ from its long boundary $\Gamma_-$ and exits from its short boundary at distance $d$ away must also then cross the core geodesic $\alpha$.  
\end{proof}

We now return to the proof of Lemma~\ref{thm:non-crossing curves shallow}.

Let $V= V^j$ be a neighborhood of the puncture on the punctured hyperbolic surface $T_0^j$. The set $V$ then contains a horocyclic neighborhood $V_0$ of that puncture, i.e. a rotationally invariant neighborhood of the cusp foliated by horocycles and bounded by a horocycle of length $\ell_0$.  We routinely parametrize $V_0$ as a strip 
${\{\Im z > \frac{1}{2\ell_-} | 0 < \Re z <1\}}$
 with vertical sides identified by the translation $z \mapsto z+1$. We are interested in precompact subneighborhoods $W_0 = \{\frac{1}{\ell_+} > \Im z > \frac{1}{\ell_-} | 0 < \Re z <1\}$ (with the same identification) embedded in $V_0$.

We next find analogues $W^k = W_j^k \subset T_j^k$ of that compact rotationally invariant neighborhood $W_0 \subset V_0 \subset T_0^j$ on which to focus. This is a variant of the well-known "plumbing" construction (see e.g. \cite{Bers}, \cite{Fay}, \cite{Mas76}, \cite {W91}) but we include some details both for the sake of completeness as well as to set our notation. 

The surface $X^k$ is constructed as the union of a pair $\{T_-^k, T_+^k\}$ of tori glued along a slit but may be uniformized as a pair of open hyperbolic tori $\{\tau_-^k, \tau_+^k\}$ glued along the geodesic which is freely homotopic to the curve defined by the slit. 
As $k \to \infty$, the length of the slit tends to zero: the flat length of the $k$th slit is short, so we may link it by an annulus of very large modulus (that tends to infinity with $n$). It is convenient to work for now with the hyperbolic surfaces $\tau_j^k$ with their geodesic boundary, but at the end of the argument, we will pass back to the flat slit surfaces $T_j^k$ with their slit boundary, using that the slits and the geodesic boundaries both tend geometrically to the puncture.

The (open) complement of the geodesic boundary is homeomorphic to the cusped torus $T_j^0$, so we may regard the hyperbolic surfaces $\tau_j^k$ (and the cusped torus $T_j^0$) as all sharing a common underlying differentiable manifold $T$, which we then equip with hyperbolic metrics $h^k$ (and $h^0$, resp.), each defined up to a diffeomorphism of $T$.  
In this setting, we may choose representatives $h^k$ and $h^0$ of those metrics so that $h^k \to h^0$, uniformly on compacta of $T$.  
In particular, there is a neighborhood $K \subset T$, which is isometrically the complement of (an isometrically embedded copy of) $V_0 \sim W_0$, in which the metrics $h^k$ converge uniformly to the metric $h^0$.  Thus, we can then find a compact annular neighborhood $W \subset T_0$ with the properties

\noindent(i) $W \subset V_0$ and is $h^0$-rotationally invariant,

\noindent(ii) the $h^k$-injectivity radius of $W$ is uniformly bounded below, as is the $h^0$-injectivity radius, (and similarly  so is the $h^k$-diameter also bounded above, as is the $h^0$-diameter), and 

\noindent(iii) there is a region $W^k \subset W$ which is $h^k$-rotationally invariant and isometric to the region $W_d = W_d(\epsilon_k)$ constructed in Lemma~\ref{lem:ShallowNoncross}. 
(In particular, $W^k$ is foliated by curves of constant geodesic curvature, is bounded by one long such curve of length $\ell_-$ and has boundary components at distance $d$ from one another.)

Note that the constant $d$ in the last condition is independent of $k$.

As the set $W$ contains $W^k$ for each $k$ and any simple geodesic which crosses both boundary curves of $W^k$ must cross the core geodesic, we see that any geodesic that crosses both boundary curves of $W$ must also cross a core geodesic. 
\end{proof}

\subsection{Twisting}
Twisting is defined both in the hyperbolic sense and with respect to the flat metric. 
Suppose one has a curve $\beta$. Lift  to the annular cover  $A(\beta)$ and denote the closed lift by $\tilde \beta$.  
In the case of the hyperbolic metric choose some perpendicular $\tilde\tau$ to $\beta$ and extend to both components of $\partial A$.  Now suppose $\gamma$ is any other curve crossing $\beta$. We take its hyperbolic geodesic representative and lift to a curve $\tilde\gamma$ on $A(\beta)$ that meets $\tilde \beta$.  Then we define the hyperbolic twist $$\twist^h_\beta(\gamma)=i(\tilde\gamma,\tilde\tau).$$ This number is well defined up to a small additive error. 
In the case of the flat metric we do the same construction now taking flat geodesic representatives and lifting. 
We denote this by $\twist^f_\beta(\gamma)$.

Suppose in the flat metric  there is no flat cylinder about $\beta$.  Equivalently, $\beta$ is represented by a union of saddle connections and, for at least one zero on $\beta$ the incoming and outgoing saddle connections make an angle greater than $\pi$.
In particular this condition holds in the case of the slits $\zeta^k$ in this paper. Suppose a geodesic $\gamma$  crosses $\beta$, and $\gamma$  does not have any saddle connections in common with $\beta$. This holds for the pair of curves $\gamma_i$
that we are interested in. 
\begin{lem} \label{lem:bounded twist}
Under the above hypothesis,  $\twist^f_\beta(\gamma)=O(1)$. 
\end{lem}

\begin{proof}

There is a saddle connection(s)  $\beta_0\subset\beta$ in some direction $\theta_0$ with the property that there  are segments $\kappa_1,\kappa_2$ leaving the  endpoints of $\beta_0$ in direction $\theta_0$ that makes an angle $\pi$ with $\beta_0$ and such that  $\kappa_i$ is  not itself  a subsegment of  $\beta$.  Since $\gamma$ has no saddle connections in common with $\beta$, if a segment of $\gamma$ crosses $\kappa_1$ in a small neighborhood of $\beta$ then it cannot simultaneously stay in that small neighborhood of $\beta$ and also cross $\kappa_2$.   Thus we can find a perpendicular to $\beta_0$ so that there is no arc of $\gamma$  which together with the perpendicular makes a closed curve homotopic to $\beta$.  This bounds the twist.  

\end{proof}

\begin{cor}
\label{twists}
Let $U$ be the collar.  Then $\ell_\rho(\gamma\cap U) \asymp  
-i(\gamma,\beta)\log \ell(\beta)$.
\end{cor}

\begin{proof}
The flat twisting about $\beta$ is $O(1)$ by the above Lemma~\ref{lem:bounded twist} and then by Theorem 4.3 of Rafi \cite{R} 
the amount of hyperbolic twisting is  $O(\frac{1}{\ell(\beta)})$.  The hyperbolic length due to twisting of $\gamma$ is then the product of the amount of twisting with $i(\gamma,\beta))\ell(\beta)$, or $O(i(\gamma,\beta))$ which is small compared to the quantity in the statement of the Lemma. 
\end{proof}

It is a well-known that the function $\ell:\T(S)\times \MF\to \R$ which assigns, to each hyperbolic surface and measured lamination, the length of the lamination is a continuous function.  Here we extend to the case of noded surfaces.
The added assumption is that lengths of the arcs of curves spent crossing collars about $\beta$ are asymptotically small.  We will consider a special case when the lamination on the limiting surface is a simple closed curve.

\begin{prop}
\label{prop:augmented}

Suppose $X^k$ is a sequence of genus $2$ surfaces  consisting of tori $T_-^k,T_+^k$ glued along a slit $\zeta^k$  with metrics $\rho^k$ converging uniformly on compact sets to $X^\infty$, a pair of rectangular tori $T_-^\infty,T_+^\infty$ with paired punctures (obtained by pinching $\zeta^k$), equipped with hyperbolic metrics $\rho_{\pm}^\infty$.   Let the almost horizontal curves for $T_{\pm}^k$ be denoted by  $\beta_{\pm}$.  As in  Lemma~\ref{thm:non-crossing curves shallow},  fix 
a neighborhood $U$ of the paired punctures associated to the pinching curves $\zeta^k$ with   the property  that if a  closed simple geodesic on $X^k$ enters $U$, then it crosses $\zeta^k$ before exiting $U$.  Let  $\gamma_k$ be a sequence of simple geodesics on $X^k$,    
converging in each $T_{\pm}^\infty$ to $\beta_{\pm}$ in the sense that there is a sequence $r_{\pm}^k$ such that $$r_{\pm}^ki(\gamma_k,\alpha)\to i(\beta_{\pm},\alpha)$$ for any $\alpha\subset T_{\pm}^\infty$ thought of as a curve in $T_{\pm}^k\setminus U$. 
Further suppose 
$$\frac{\ell_{\rho^k}(\gamma_k\cap U)}{\ell_{\rho^k}(\gamma_k\cap (X^k\setminus U))}\to 0.$$
Then for the almost vertical closed curve $\sigma_{\pm}^k$  in $T_{\pm}^k$
$$\ell_{\rho^k}(\gamma_k)\sim \sum_{i=\pm}i(\gamma_k,\sigma_{\pm}^k)\ell_{\rho_{\pm}^\infty}(\beta_{\pm}).$$
\end{prop}

\begin{proof}
Consider a connected component $\Gamma_i = \Gamma_{-,i}^k = \gamma_k \cap (X^k \setminus U)$ of $\gamma_k$ on the $T_-^k$ toral component of 
$X^k \setminus U$. A similar construction will work for the components of $\gamma_k$ on $T_+^k$, and since it will not matter which component we choose, we will denote the component of interest by just the symbol $\Gamma_i$, leaving it understood that this is a subarc of $\gamma_k$ on $T_-^k$.

The arc $\Gamma_i$ has endpoints on $T_-^k \cap \partial U$, and so is homotopic, rel $\partial U$, to a curve of slope $(m_i, n_i)$, written here in terms of the basis $\beta_-^k$ and $\sigma_-^k$. Indeed, because the curve $\gamma_k$ is a simple closed curve, if $\Gamma_j^k$ is a different component of 
$\gamma_k \cap (X^k \setminus U)$ of $\gamma_k$, then $(m_j, n_j) = (m_i, n_i)$, so we might as well denote the slopes of both $\Gamma_i$ and $\Gamma_j$ by $(m,n) = (m_-, n_-)$.

Our first displayed hypothesis implies that $\frac{n_-}{m_-} \to 0$ (and also that $\frac{n_+}{m_+} \to 0$), so in particular $m_{\pm} >> n_{\pm}$. Now, we may lift the hyperbolic one-holed torus $T_-^k$ to $\HH^2$, with the holonomic image of the elements $[\beta]$ and $[\sigma]$ being given by hyperbolic isometries $B$ and $S$, respectively.  A word in this (free) fundamental group is given by a word in $B$ and $S$, but our conclusion that $m >> n$ implies that the element $[\Gamma_i]$ is represented by a word of the form 
$$B^{b_1}SB^{b_2}SB^{b_3}S....B^{b_l}S^s,$$ 
where $|b_i -b_j| \le 1$, $|b_i - [\frac{m_{\pm}}{n_{\pm}}]| \le 1$ and the final exponent $s$ is either $0$ or $1$: the point of this representation is that the arc $\Gamma_i$ may be thought of as first traversing $\beta$ many times, roughly $[\frac{m_{\pm}}{n_{\pm}}]$ times, up to an error of $\pm 1$, before adding a single iteration of $\sigma$, and then starting over.

Let $p$ be the initial point of $\Gamma$ (so that $p \in \partial U \cap T_-^k$). Then the terminal point of $\Gamma$ is within a bounded distance (say twice the diameter of $T_-^{\infty}$) of $B^{b_1}SB^{b_2}SB^{b_3}S....B^{b_l}S^s(p)$.  We set about to compute this distance.

We begin by noticing that, for $b_j$ large, the distance $d_{\HH^2}(q, B^{b_j}S(q))$ is asymptotic to $b_j\ell_{\rho^k}(\beta)$: this is because the geodesic connecting the points $q$ and $B^{b_j}S(q)$ is within a universally bounded distance of a broken geodesic consisting of following a lift $\tilde{\beta}$ for $b_j$ iterations of its generator, followed by a single lift $\tilde{\sigma}$ of $\sigma$, so that the hyperbolic law of cosines yields the asymptotic statement that the ratio goes to $1$. Then we note that $B^{b_1}SB^{b_2}SB^{b_3}S....B^{b_l}S^s$ is nearly $l$ iterations of $B^{b_j}S$ up to at most $l$ occasional insertions or deletions of a single $B$ or $S$.  Thus, the product of these elements is asymptotically additive for length, so
\begin{align*}
d_{\HH^2}(p, B^{b_1}SB^{b_2}SB^{b_3}S....B^{b_l}S^s(p)) &= \sum_j b_j\ell_{\rho^k}(\beta) + O(l) \\
&\sim i(\Gamma, \sigma)\ell_{\rho^k}(\beta).
\end{align*}
Of course, the absolutely bounded constants are unimportant in the limit, so we find that 
$$\ell_{\rho^k}(\Gamma) \sim d_{\HH^2}(p, B^{b_1}SB^{b_2}SB^{b_3}S....B^{b_l}S^s(p)) \sim i(\Gamma, \sigma)\ell_{\rho^k}(\beta)$$.

Finally we add up the lengths of the components of $\gamma_k$ to find that 
\begin{align*}
\ell_{\rho^k}(\gamma_k) &= \ell_{\rho^k}(\gamma \cap (X^k \setminus U)) + \ell_{\rho^k}(\gamma_k \cap U)\\
&\sim \ell_{\rho^k}(\gamma_k \cap U) \text{ because of the second displayed hypothesis}\\
&\sim \sum_{i, \pm} i(\Gamma_i, \sigma_{\pm}^k) \ell_{\rho^k}(\beta_{\pm}^k)\\
&=\sum_{\pm} i(\gamma_k, \sigma_{\pm}^{k}) \ell_{\rho^k}(\beta_{\pm}^{k})\\
&\sim \sum_{\pm} i(\gamma_k, \sigma_{\pm}^{k}) \ell_{\rho_{\pm}^{\infty}}(\beta_{\pm}),
\end{align*}
as desired.

\end{proof}

Much of our effort will be spent estimating hyperbolic lengths, and in many cases, Proposition~\ref{prop:augmented} will be sufficient.  We will have some cases however, in which the limiting torus will also have a short curve, and for those cases, we will need a refinement of Proposition~\ref{prop:augmented}.

\begin{lem}
\label{lem:asymptotic} Let $(T, \rho)$ be the complement of a cusp neighborhood of a once-punctured torus.  We assume that $T$ has a short curve $\alpha$.  Let $\tau$ be a curve which meets $\alpha$ and has least length among all such intersecting curves. Suppose that $\gamma$ is a (long) simple geodesic arc with endpoints on $\partial T$. Then
\begin{equation}
\label{eq:asymptotic}
\ell_{\rho}(\gamma) \sim i(\gamma, \alpha)(-2\log \ell_{\rho}(\alpha)) + i(\gamma, \tau)\ell_{\rho}(\alpha)
\end{equation}
\end{lem}

\begin{rem} Of course, for the applications we have in mind, we imagine $\gamma$ as being a component of the intersection of a closed geodesic with $T$.
\end{rem}

\begin{proof}
Let $U$ be a collar neighborhood of $\alpha$ as we have previously constructed in Lemma~\ref{lem:ShallowNoncross}, so that any component of $\gamma \cap U$ also meets $\alpha$.

Then the compact set $T \setminus U$ is a three-holed sphere $S$ of diameter bounded independently of the length of $\alpha$.  In particular, because $\partial\gamma$ lies on a single component of $\partial   S$, the set $\gamma \cap S$ consists exactly of two arcs which meet $\partial T \subset \partial S$, and multiple arcs, say N such arcs, which connect a component of $\partial U$ to another component of $\partial U$. Of course, by the construction of the neighborhood $U$, once $\gamma$ meets $\partial U$, it then continues across $U$, intersecting $\alpha$ enroute; thus we see that 
\begin{equation} \notag
N= i(\gamma, \alpha) -1.
\end{equation}

We classify the topological type of those arcs that connect $\partial U$ to $\partial U$.  

Consider first an arc $A$ in $S$ that connects a component $\partial_0 U$ of $\partial U$ to itself.  Then, as a three-holed sphere is the double of a hexagon, glued along their odd sides, and the simplicity of $A$ precludes it having two segments in the same hexagon which pass between the same two odd edges, we see that there are only a finite number of possible homotopy classes of arcs $A$, components of $\gamma \cap S$ of the simple curve $\gamma$.  Thus taking a maximum over geodesic representatives of these homotopy classes over the compact set of possible endpoints and over the compact moduli spaces of these particular three-holed spheres of bounded geometry yields an upper bound on the possible lengths of arcs $A$ in $S$ that connect a component $\partial_0 U$ of $\partial U$ to itself. 

Next we consider the arcs which connect a boundary component $\partial_0 U$ of $\partial U$ to the other boundary component, say $\partial_1 U$, of $\partial U$.    We can imagine the three-holed sphere $S$ here as being a one-holed cylinder, with the cylindrical part bounded 
by $\partial U$.  From this point of view, there is a shortest arc $B$ across the cylinder, and any component of $\gamma \cap S$ that connects the opposite components of $\partial U$ in $S$ just twists around the cylinder, with its twist reflected in intersections with $B$.  

We want to imagine these twists as occurring within $U$ so, informally, we expand $U$ to include them. To see that this is possible with adding only a finite diameter neighborhood to $U$, consider a hyperbolic geometric model for the three-holed sphere $S$ as follows.  Begin with a vertical strip, say $\{|x| < 1\}$, in the upper half plane model for $\HH^2$. Intersect that strip with a hyperbolic halfplane bounded by a geodesic, say $\delta$, whose endpoints are at $\pm d \in \R$, where $d<1$ is a number near $1$. Connect the geodesic $\delta$ to the edges of the strip by a pair of geodesics -- those geodesics, say $\alpha_-$ and $\alpha_+$, will meet the real line in pairs of points that are either both positive or both negative.  Cutting off this domain in the strip by non-geodesic constant curvature curves with the same endpoints as $\alpha_-$ and $\alpha_+$ but which (as non-geodesics) intrude further into the strip and also by a horocycle $\{y = d^*\}$ yields a domain shaped like a hexagon. This hexagon doubles to give a model for $S$ (incorporating its symmetries).  Then note that we can choose those constant geodesic curves both symmetrically and also so large that they touch each other along the $y$-axis, and we can then also bring the horocycle down to where those curves meet the edge of the strip. (Note that the constant geodesic curves maximize their $y$-value on the edge of the strip, as the geodesic -- and hence each of these curves -- meets this edge orthogonally.)This new domain gives a model for $S$ where any arc which connects $\partial_0 U$ of $\partial U$ to the other boundary component $\partial_1 U$ of $\partial U$ has uniformly bounded length -- this is because the model for $S$ is now the double of the remaining triangle in the strip.

This discussion above then serves to bound the length of the components of $\gamma \cap S$ and to estimate the number of such components.





It remains to consider a component, say $\Gamma$, of $\gamma \cap U$.  We can imagine that $\Gamma$ meets $\tau \cap U$ in $p$ crossings. (If the segment $A$ of $\gamma \cap S$ that preceded $\Gamma$ had $p'$ inessential crossings of $\tau \cap S$ with $\gamma \cap S$, we imagine -- proceeding as we did in the paragraph which constructed a hyperbolic geometric model for $S$ -- extending $U$ by a bounded amount (independent of the length of the short curve $\ell_{\rho}(\alpha)$ or the diameter of $U$) to capture all of those $p'$ inessential crossings of $A$ with $\tau$ (rel $\partial U$): this will simplify the accounting while not affecting the terms in the estimate.) 

We estimate the length of $\Gamma$. It will be convenient to estimate the length of a curve that is a slightly perturbed version of $\Gamma$, but whose length is within a fixed constant of that of $\Gamma$. To begin, in the standard collar $U$, consider a geodesic diameter $\tau_0$ of $U$ which meets the core curve $\alpha$ orthogonally; we may choose $\tau_0$ so that it shares an endpoint with $\Gamma$. Then connect the other endpoint of $\Gamma$ to the other endpoint of $\tau_0$ by an arc in the boundary $\partial U$ of $U$ of at most half the length of $\partial U$; we call this new arc $\Gamma'$.  Let $\Gamma_0$ be the geodesic arc in $U$ which is homotopic to $\Gamma'$: it is easy hyperbolic geometry that $\Gamma$ and $\Gamma_0$ have lengths differing by at most a fixed constant. In addition, since $\tau_0$ intersects $\tau$ at most a bounded number of times, and $\Gamma_0$ is homotopic to $\Gamma$ up to a half turn of $\partial U$, we see that 

\begin{equation} \label{eqn: intersection bound}
|i(\Gamma, \tau) - i(\Gamma_0, \tau_0)| = O(1)
\end{equation}

On the other hand, the arc $\Gamma_0$ is preserved by the reflection in $U$ across the core geodesic $\alpha$, and so the length of $\Gamma_0$ is given as twice the length of its restriction to one component, say $U_-$, of $U \setminus\alpha$.  Lifting to the universal cover $\tilde{U_-}$ (and using the natural notation for the arcs), we see that $\tilde{\Gamma_0}$ is the hypotenuse of a hyperbolic right triangle with legs $\tilde{\tau_0}$ and a portion, say $\tilde{\alpha'}$, of  $\tilde{\alpha}$ of length $i(\Gamma_0, \tau_0) \ell_{\rho}(\alpha))$.  

It is now easy to estimate the length of $\tilde{\Gamma_0}$, since the hyperbolic Pythagorean theorem tells us that 
\begin{equation} \label{eqn: Pyth}
\cosh \ell_{\rho} (\tilde{\Gamma_0}) = \cosh  \ell_{\rho} (\tilde{\tau_0}) \cdot  \cosh  \ell_{\rho} (\tilde{\alpha'}).
\end{equation}
If all three lengths are long, we conclude from \eqref{eqn: Pyth} that
\begin{equation*}
\ell_{\rho} (\tilde{\Gamma_0}) \sim\ell_{\rho} (\tilde{\tau_0}) + \ell_{\rho} (\tilde{\alpha'}),
\end{equation*}
and so by the definition of $\tilde{\alpha'}$ and the bounds \eqref{eqn: intersection bound} on intersection numbers, we find that
\begin{align*}
\ell_{\rho}(\Gamma) &\sim \ell_{\rho}(\tau_0) + i(\Gamma, \tau) \ell_{\rho}(\alpha)\\
&\sim -2\log \ell_{\rho}(\alpha) + i(\Gamma, \tau) \ell_{\rho}(\alpha).
\end{align*}



 Putting these estimates together, we see that, summing over components $A$ of $\gamma \cap S$ and components $\Gamma_i$ of $\gamma \cap U$, we have 
\begin{align*}
\ell_{\rho}(\gamma) &= \sum_A \ell_{\rho}(A)  + \sum_{\Gamma_i} \ell_{\rho}(\Gamma_i) \\
&\sim N + \sum_{\Gamma_i} \{i(\Gamma_i, \alpha)(-2\log \ell_{\rho}(\alpha)) + i(\Gamma_i, \tau)\ell_{\rho}(\alpha)\}\\
&\sim[i(\gamma, \alpha) + 1] +  i(\gamma, \alpha)(-2\log \ell_{\rho}(\alpha)) + (\sum_{\Gamma_i}i(\Gamma_i, \tau))\ell_{\rho}(\alpha).
\end{align*}
 We have already noted that $| i(\gamma, \tau) - \sum_{\Gamma_i} i(\Gamma_i, \tau)| \le N$, with a small ambiguity of whether a crossing of $\gamma$ between one of the components of of $\partial U$ and itself happens in $S$ or not. Thus, 
\begin{equation*}
|\sum_{\Gamma_i} i(\Gamma_i, \tau) \ell_{\rho}(\alpha) - i(\gamma, \tau)\ell_{\rho}(\alpha)| \le N\ell_{\rho}(\alpha) \sim i(\gamma, \alpha)\ell_{\rho}(\alpha).
\end{equation*}
As we are assuming that $\ell_{\rho}(\alpha)$ is small, especially in comparison to $-2\log \ell_{\rho}(\alpha)$, we see that 
\begin{align*}
\ell_{\rho}(\gamma) &\sim [i(\gamma, \alpha) + 1] +  i(\gamma, \alpha)(-2\log \ell_{\rho}(\alpha)) + i(\gamma, \tau)\ell_{\rho}(\alpha)\\
&\sim i(\gamma, \alpha)(-2\log \ell_{\rho}(\alpha)) + i(\gamma, \tau))\ell_{\rho}(\alpha),
\end{align*}
as desired.

\end{proof}

Having shown that some of the geometric limits of the flow will be noded surfaces with rectangular tori components, we pause to record some of the hyperbolic geometry of these surfaces.

\begin{prop}
\label{thm:increasing}
In the once punctured rectangular $(a,\frac 1 a)$-torus $T_a$, the hyperbolic length of the geodesic in the horizontal homotopy class is increasing with $a$.
\end{prop}

\begin{proof}
Given $a,b$ without loss of generality assume  $b>a$.   We use the harmonic map $F:T_a\to T_b$ with Hopf differential $\Phi$, a quadratic differential. Since $\Phi$ is defined on the punctured torus,  we may write $\Phi=e^{i\theta} dz^2$  for some $\theta$. Let $\gamma$ be the hyperbolic geodesic crossing horizontally.  One easily computes (see e.g.\cite{Yamada}, \cite{WHessian}) that the 
derivative of length is given by $$\int_\gamma Re(\Phi)/g_0$$ where $g_0$ is the hyperbolic metric.  To show this derivative is positive  we wish to prove $\theta=0$. 
Consider the reflections $r_a,r_b$ about the vertical central curves of $T_a,T_b$, each of which are geodesics in the hyperbolic metrics.   Then $r_b\circ F\circ r_a$ has the same energy as $F$ since $r_a,r_b$ are isometries. Since the hyperbolic surfaces are negatively curved, the harmonic map is unique, and so $F=r_b\circ F\circ r_a$. We conclude that $F$ preserves the horizontal and vertical frame along the geodesic.   Since the map  expands along the horizontal trajectories of the Hopf differential, we conclude that $\theta$ is either $0$ or $\pi$. 

Consider now a quadratic differential $\Phi = cdz^2$ on a rectangular surface; we are interested in the sign of $c$. For $c$ real, the leaf space of the horizontal trajectories is a circle of length 1, and the energy of the projection map is $E =  (1/2) \Ext_h$, where $\Ext_h$ is the extremal length of the horizontal curve.  Now the gradient in Teichm\"uller space of either $E$ or $\frac{1}{2} \Ext_h$ is $-2 \Psi$ (see \cite{Wentworth:Gardiner} or \cite{WExtremal}), where $\Psi$ is the Hopf differential of the harmonic map from the surface to the target circle, i.e. the leaf space of the horizontal foliation.  That the gradient has this form is commonly known as Gardiner's formula 
\cite{Gardiner:MinimalNorm}, \cite{Gardiner:TeichTheoryBook} and written as an expression for the the derivative of extremal length as
\begin{equation*}
d Ext_h [\mu] = -2 \Re \int \Psi \mu
\end{equation*}
where $\mu$ is an (infinitesimal) Beltrami differential.

Because the minimal stretch of this harmonic map to the circle -- as a projection along the horizontal leaves --  is clearly along those leaves (since each such leaf is projected to a point, making the stretch in that direction vanish, hence the minimum possible), the minimal stretch direction for this map to the leaf space is vertical for the Hopf differential $\Psi$; since that direction is also horizontal for $\Phi$, we may conclude that $\Psi = -k \Phi$ , where $k>0$.
Thus the derivative of extremal length may be written in terms of $\Phi$ as
\begin{align*}
d Ext_h [\mu] &= -2 \Re \int \Psi \mu\\
&=  +2 k \Re \int \Phi \mu 
\end{align*}
where again $\mu$ is an (infinitesimal) Beltrami differential and $k>0$. 

We apply this formula to the deformation for the Riemann surface which stretches it from having width $a$ to having width $b$.   An infinitesimal deformation in that direction may be represented as  $\mu= \bar{d} \frac{d\bar{z}}{dz}$ where $d>0$: this representation follows by explicitly computing the Beltrami differential for an affine stretch between rectangular tori of identical heights but where the widths of the range is $b$ and the width of the domain is $a<b$.

For such a stretch, because the extremal length of the horizontal curve on the domain is $a$, while  the extremal length of that same curve system on the range is $b>a$, we see that $dExt_h [\mu] > 0$.

On the other hand, applying the displayed formula above and substituting our expression for $\mu$, we compute:
\begin{align*}
0< dExt_h [\mu] &= 2 k \Re \int \Phi \bar{d}\\
                    &= 2 \Re \int kc \bar{d}
\end{align*}
                 
Thus we conclude that, since $d>0$ and $k> 0$, then so is $c>0$, or alternatively that $\theta =0$, which is what we required in the first paragraph to conclude that the hyperbolic length of the horizontal homotopy class was increasing.



\end{proof}

\begin{prop}
\label{lem:largecylinder} There exists $D$ such that the hyperbolic length of the horizontal geodesic in the once punctured rectangular $(a, \frac 1 a)$-torus is between $4\log(a)-D$ and $4\log(a)+D$.
\end{prop}
\begin{proof}

First note that the $(1,1)$ punctured (flat) rectangular torus is conformally equivalent to the 
$(L,L)$ punctured (flat) rectangular torus, where $L$ indicates the length of the vertical 
geodesic on the uniformized $(1,1)$ punctured rectangular torus. Then we see that if we 
graft a cylinder of length $\lambda$ along that vertical geodesic, we obtain a torus 
conformally equivalent to the $(L + \lambda, L)$ punctured rectangular torus, hence also 
equivalent to the $(a, \frac{1}{a})$ punctured rectangular torus for $a= \sqrt{1+ 
\frac{\lambda}{L}}$.

On such a $(L + \lambda, L)$ punctured rectangular torus, Theorem~6.6 of \cite{DumasWolf} 
estimates the hyperbolic length $\ell_v$ of the vertical geodesic as $\ell_v = \ell_v(a) = 
\ell_v(L,\lambda) =  \frac{\pi L}{\lambda} + O(\lambda^{-2}) = \frac{\pi}{a^2-1} + O(a^{-4})$.

Letting $a \to \infty$, these $(a, \frac{1}{a})$ punctured rectangular tori converge to 
noded hyperbolic tori (i.e. a hyperbolic triply punctured sphere, where two of the 
punctures are paired); one way to see this is to represent the $(a, \frac{1}{a})$ 
punctured rectangular tori as a family of metrics on a smooth punctured torus, and then 
observe that the metric tensors converge, uniformly on compacta in the complement of the 
short vertical geodesic. From this convergence, we see that for any $\epsilon>0$ and $a$ 
sufficiently large (depending on $\epsilon$), the $(a, \frac{1}{a})$ punctured rectangular 
torus may be decomposed as a disjoint union of a (rotationally invariant) hyperbolic 
cylinder with boundary arcs of a fixed length $C_0$ and core geodesic of length 
$\ell_v(a)$, together with a compact punctured sphere with two boundary components of 
length $C_0$: here all of the compact pieces are uniformly $1+\epsilon$ quasi-isometric, 
independent of the choice of $a$ (sufficiently large).

In particular, because a hyperbolic cylinder with core geodesic $\ell_v(a)$ and boundary 
of length $C_0$ has diameter $D(a,C_0) = 2 \log \frac{C}{\ell_v(a)}+O(1)$, we see that the 
horizontal geodesic of the $(a, \frac{1}{a})$ punctured rectangular torus is of length $2 
\log \frac{C}{\ell_v(a)}+O(1) = 4 \log a +O(1)$, once we apply above our formula for 
$\ell_v(a)$.  This concludes the proof of the lemma.

\end{proof}

\begin{lem}
\label{lem:extremal}
Suppose  $T_-,T_+$ are tori glued along  an almost vertical  slit $\zeta$ of length  $\delta$ . The total glued surface has area $2$ with the area of $T_+$ at least $1$.   Suppose  $T_-$ has an almost vertical 
side $\sigma_-$ of length $b$ and the other curve $\beta_-$  is not almost vertical and its horizontal component  has length  $\delta$.  Assume $\delta<1/b$. Assume also in the standard basis for $T_+$  there is an almst vertical curve $\sigma_+$ of length $b$ and the other basis curve $\beta_+$ has horizontal component comparable to $1/b$.  
Then \begin{itemize}
\item the hyperbolic length of $\beta_-$  is comparable to  $\delta/b$.
\item if there is another not almost vertical slit $\zeta'$ that crosses the first and has length  $\delta\leq \kappa=O(1)$, then  the hyperbolic length of $\zeta$ is comparable to $\frac{1}{\log(\kappa/\delta)}$.
\end{itemize} 
\end{lem}


\begin{proof}
For curves of small hyperbolic length, the hyperbolic length is comparable to the extremal length \cite{Maskit}.
To find a lower bound for extremal length, we build a metric $\rho$ which is the given  flat metric on 
$T_-$  and extended to $T_+$ in a  neighborhood of the slit defined so that the intersection of that neighborhood with $T_+$ is a rectangular region of vertical length $b$ and horizontal length $\delta$ (hence has area comparable to $\delta b$). On the complement of this neighborhood in $T_+$ we define  $\rho$ to be $0$.  Then in the $\rho$ metric the length of any curve in the homotopy class of $\beta_-$ is at least $\delta$ and the area is comparable to $\delta b$ so the extremal length is at least a number comparable to $\delta^2/\delta b=\delta/b$.    For the upper bound  we use that the extremal length  is the reciprocal of the modulus of the biggest cylinder in the homotopy class.  Since $\beta_-$ is not almost vertical and $\sigma_-$is almost vertical, the cylinder in $T_-$ in the class of $\beta_-$ has modulus at least comparable to  $b/\delta$.   

The second statement is the first conclusion of Lemma 2.2 of \cite{R}.
\end{proof}

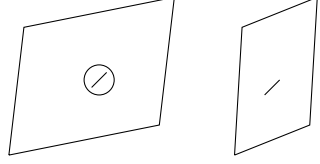
\begin{figure}[h]
\caption{The metric is 1 on the second parallelogram and in the circle on the first parallelogram and 0 elsewhere}
\begin{tikzpicture}
\draw (0,0) --(2.0,0.4)--(2.2,2.1)--(0.2,1.7)--(0,0);
\path (1.1,.9) edge (1.3,1.1);
\draw (3,0)--(4,0.4)--(4.1,2.1)--(3.1,1.7)--(3,0);
\path (3.4,.8)edge(3.6,1);
\draw (1.2,1) circle (0.20);
\end{tikzpicture}
\end{figure}

\section{Proof of Theorems~\ref{thm:CF2}, \ref{thm:CF3}, \ref{thm:CF4}}.

\subsection{Proof of Theorem \ref{thm:CF2}}

We choose $c \notin \{-1,0,1\}$ and consider $g_t(X_c, \omega_c)$.
We first show that the barycenter $[F,\mu_{0}]$ is a limit  point.  Since the set of limit points is closed, it is enough to show that there are limit points $[F,\mu_\tau]$ for $\tau$ arbitrarily close to $0$.  For a given $a$, consider the  rectangular punctured tori $T_-^\infty,T_+^\infty$ with the same vertical side lengths and such that  $T_+^\infty$ has flat horizontal length $a$ and $T_-^\infty$ has flat horizontal length 
$a \frac{c}{1-c}$.
By Proposition~\ref{lem:largecylinder}, for  $a$ sufficiently large,  their hyperbolic metrics $\rho_-,\rho_+$ are such that $\frac{\rho_-(\beta_-)}{\rho_+(\beta_+)}$ is arbitrarily close to $1$,  where $\beta_-,\beta_+$ are the horizontal curves $\beta_{\pm}= \lim_{k \to \infty} \beta_{\pm}^k$.

 By Proposition~\ref{prop:geom} and Proposition~\ref{prop:geom2},  there  are   times $t_k$ with  $$e^{t_k}\asymp aq_{n_k}$$ so that   $g_{t_k}r_{\theta_c} (Y,\omega_c)$ consists of tori $T_-^k,T_+^k$ glued along the $k^{th}$ slit and that converge to $T_-^\infty,T_+^\infty$.

On the surface $g_{t_k}r_{\theta_c} (Y,\omega_c)$ 
the curves $\gamma_i$  are almost horizontal. 
We wish to apply Proposition~\ref{prop:augmented} to the curves $\gamma_i$.  We check the hypotheses of that Proposition.

  Since the slope of $\gamma_i$ in the flat structures goes to $0$, the  first condition  holds in 
Proposition~\ref{prop:augmented}. 
Now we check the second condition of Proposition~\ref{prop:augmented}.
The length $|\zeta^k|_{t_k}$ of the $k^{th}$ slit $\zeta^k$ with respect to the flat metric of $g_{t_k}r_{\theta}(Y,\omega_c)$ satifies $$|\zeta^k|_{t_k}\asymp \frac{a}{a_{n_k+1}}.$$  Its extremal length is larger than the square of its flat length, and for short curves, extremal length is comparable to hyperbolic length.  Thus 
$$\ell_{\rho_{t_k}}(\zeta^k)\geq c(\frac{a}{a_{n_k+1}})^2,$$ for some $c>0$.  

The third conclusion of Proposition~\ref{prop:geom} says that the number of times that $\gamma_i$ intersects $\zeta^k$ is $$i(\gamma_i,\zeta^k)\asymp q_{n_{k-1}}.$$ Thus by Corollary~\ref{twists} 
$$\ell_{\rho_{t_k}}(\gamma_i\cap U)\asymp  q_{n_{k-1}}(-\log a_{n_k+1}-\log a).$$  On the other hand  the horizontal distance across
the 
 tori  $T_{\pm}^k$  is comparable to $1$ and the number of times $\gamma_i$  crosses the tori is comparable to $q_{n_k}$,  so that the total length of these arcs is comparable to $q_{n_k}$.  But since $a$ is fixed, Condition~\ref{eq:condC}  says that 
 $$\lim_{k\to\infty} \frac{\ell_{\rho_{t_k}}(\gamma_i\cap U)}{\ell_{\rho_{t_k}}(\gamma_i\cap (T^k_{\pm} \cap U))}=\lim_{k\to\infty}\frac{-q_{n_{k-1}}\log a_{n_k+1}}{q_{n_k}}\to 0.$$ 
This verifies the second condition   of Proposition~\ref{prop:augmented}.
We conclude 
\begin{equation}
\label{eq:lengthsratio}
\frac{\ell_{\rho_{t_k}}(\gamma_1)}{\ell_{\rho_{t_k}}(\gamma_2)}- 
\frac{i(\gamma_1,\sigma_-^k)\ell_{\rho_-}(\beta_-)+i(\gamma_1,\sigma_+^k)
\ell_{\rho_+}(\beta_+)}{i(\gamma_2,\sigma_-^k)\ell_{\rho_-}(\beta_+)+i(\gamma_2,\sigma_+^k)\ell_{\rho_+}(\beta_+)}\to 0.
\end{equation}
Since  $\frac{\ell_{\rho_-}(\beta_-)}{\ell_{\rho_+}(\beta_+)}$ is close to $1$,
and $i(\gamma_1,\sigma_-^k)= i(\gamma_2,\sigma_+^k)$ while $i(\gamma_2,\sigma_-^k)=i(\gamma_1,\sigma_+^k)$ by Lemma~\ref{lem:crossings}, it follows that  
$\frac{\ell_{\rho_{t_k}}(\gamma_1)}{\ell_{\rho_{t_k}}(\gamma_2)}$ is close to $1$ so the barycenter is a limit.

Now we similarly show that we also find limit points that are not the barycenter.  We use rectangular  punctured tori $T_-^\infty,T_+^\infty$ as limiting surfaces  with horizontal curves $\beta_-$, $\beta_+$ with fixed but unequal flat lengths, which we can assume without loss of generality satisfy  $$|\beta_-|>|\beta_+|.$$    Proposition~\ref{thm:increasing}  then says that their hyperbolic lengths satisfy $$\ell_{\rho_-} (\beta_-)>\ell_{\rho_+}(\beta_+).$$ 
We then see from (\ref{eq:lengthsratio}) (and again using Lemma~\ref{lem:crossings}) that   
$$\lim_{k\to\infty}\frac{\ell_{\rho_{t_k}}(\gamma_1)}{\ell_{\rho_{t_k}}(\gamma_2)}>1,$$
which says that the limit is not the barycenter. 


 \subsection{Sketch of proof of Theorem~\ref{thm:CF3}, Theorem~\ref{thm:CF4}}

In the ergodic case the tori have areas that have unbounded ratios by Lemma~\ref{lem:erg}.  Consequently it is possible that one of the tori can live in the compact set of tori while the other may have a short curve. 

In proving Theorem~\ref{thm:CF3} we need to find {\em some} times where we are near the ergodic endpoints and in Theorem~\ref{thm:CF4} to show that this happens for {\em all} large times. The strategy will be to show that the hypothesis of  Lemma~\ref{lem:ergodiclimit} holds. This hypothesis requires that the limit of the ratio of the  lengths of $\gamma_2$ to $\gamma_1$  can be expressed  just in terms of intersections with curves converging to the desired ergodic endpoint $[F,\mu_+]$.  In the case at hand these latter curves are the curves $\sigma_+^k,\beta_+^k$ that are contained in the torus $T_+^k$. One therefore needs to show that the terms that involve  intersections of $\gamma_i$ with curves in $T_-^k$ should not appear asymptotically in the ratio of lengths of the $\gamma_i$. The formulas 
for lengths of $\gamma_i$ are obtained in Proposition~\ref{prop:augmented} and Lemma~\ref{lem:asymptotic} and so, to apply the formula in Lemma~\ref{lem:ergodiclimit}, we need to show that lengths of $\gamma_i$ restricted to  $T_+^k$ dominate lengths of $\gamma_i$ restricted to  $T_-^k$. We will need to show either that there are fewer intersections of the $\gamma_i$ with curves in $T_-^k$ than in $T_+^k$ or that lengths of each arc is smaller. Both of these possibilities  will follow from the fact that ergodicity of the initial flow implies that the area of $T_-^k$ is much smaller than the area of $T_+^k$.  Since the vertical components of closed curves do not change, what occurs is that the horizontal components become much smaller, so even though we have little control over the shape of the smaller rectangle $T_-^k$, that certain lengths are so much smaller allows us to conclude that the hyperbolic lengths in the larger torus dominate those of the smaller rectangle. Thus, in this first part of the proof of Theorem~\ref{thm:CF3}, because of this length estimate and because we will only need to estimate the geometry along a single sequence of times, our argument will depend just on the choice of  $k^2$ in the continued fraction expansion. 

To prove that the barycenter is a limit point, even though again we are only interested in subsequences of times along the flow, we will also need to choose the other terms $a_{4k+2}$ and $a_{4k+4}$, and make those much larger than $k^2$. For this argument, we use that we can deform the limits we obtained in the first part to ones that are stretched somewhat horizontally.  Indeed, we can stretch those limiting rectangles so much that the hyperbolic lengths of the horizontals, even though they are taken on surfaces which still have very different flat areas, are approximately the same length. In short, we overcome the vast differences in flat areas by a long stretch of the approximates (while still keeping the relevant slits small enough so that the surfaces are nearly pinched).

In Theorem~\ref{thm:CF4} one needs to show that the only limit is the ergodic endpoint.  Of course, this will be accomplished again by the  choice of the continued fraction, but in this case, the choice is more delicate, as we have to control the limits of all subsequences -- not only obtaining the desired limit point but making sure that there are no stretches such as those in the second part of Theorem~\ref{thm:CF3} which would yield an interior point of the possible limiting simplex.  In particular, we cannot allow larger horizontal stretches of our approximates such as those we designed in the second part of Theorem~\ref{thm:CF3}. In outline, the argument follows that for Theorem~\ref{thm:CF3}: we aim to use Lemma~\ref{lem:ergodiclimit} to conclude that  only the ergodic endpoint appears as a limit, and this requires us to show that the lengths in the larger torus $T_+^k$ dominate those in the smaller torus $T_-^k$. To accomplish those estimates, we use the formulae from Lemma~\ref{lem:asymptotic} and Proposition~\ref{prop:augmented} to restrict our attention either to the amount of twist of a curve $\gamma_i$ or the length that it picks up while crossing a short collar. Now, because in different regimes of time for $t$, the lengths in the two different tori are dominated by different geometric types, the argument breaks up into three cases.  In each of the three cases, the basic estimate needed to invoke Lemma~\ref{lem:ergodiclimit} is achieved by different dominant lengths.

\subsection{Proof of Theorem~\ref{thm:CF3}.} 

The accumulation set is clearly an interval, as the trajectories are proper in \Teich space. Thus we need only to find a sequence of times which approach the ergodic endpoint and a sequence of times that approach the barycenter.  We begin with the approach to the ergodic endpoint.

\begin{prop}\label{prop:sec2 ergpt} Along the sequence of times  $$t_k=\frac{1}{2}\log(q_{4k+5}^2+p_{4k+5}^2)$$ 
 we get the limit point which is the ergodic endpoint $[F,\mu_+]$. 
\end{prop}

We begin by recalling our discussion of the estimate \eqref{tinytorus}: the surface $g_{t_k}r_\theta(Y,\omega)$ is within $O(\frac{1}{a_{4k+5}}+\frac{1}{a_{4k+6}})$  of two glued  tori with vertical side of flat length comparable to one. Let $\sigma_i^k$ denote the almost-vertical sides and $\beta_i^k$ the now almost-horizontal curves that come from the penultimate convergent:  see Definition~\ref{defn:tori}.

The flat length $|\beta_+^k|_{t_k}\asymp 1$.
We wish to show that 
the term  $$i(\gamma_i,\sigma_+^k) \ell_{\rho_{t_k}}(\beta_+^k)\asymp q_{4k+5}$$  in Proposition~\ref{prop:augmented}  dominates  the  terms coming from lengths in $T_-^k$. There are three terms contributing to the length of $\gamma_i$ in $T_-^k$: crossing $\beta_-^k$, crossing $\sigma_-^k$ and the collar. The first two will be treated in Equation \eqref{eq:5 bound1} and the last will be covered by Lemma \ref{lem:collar}.

By the estimate \eqref{tinytorus},    $$\frac{1}{4k^2}\leq |\beta_-^k|_{t_k}\leq \sum_{j=k}^{\infty}\frac{1}{j^2}\asymp \frac 1 k.$$  Invoking Lemma~\ref{lem:extremal}, 
then yields
$$\frac{1}{16k^4}\leq \ell_{\rho_{t_k}}(\beta_-^k)\leq O(1/k^2).$$
 Then  the term   appearing in the right side of Lemma~\ref{lem:asymptotic} coming from lengths in $T_-^k$,  
  \begin{equation}\label{eq:5 bound1}i(\gamma_i,\sigma_-^k)\ell_{\rho_{t_k}}(\beta_-^k)-2i(\gamma_i,\beta_-^k)\log \ell_{\rho_{t_k}}(\beta_-^k)=o(q_{4k+5}).
  \end{equation}
   This is 
because, from Lemma~\ref{lem:crossings}, we have that  $i(\gamma_i,\sigma_-^k) = O(q_{4k+5})$, while $\ell_{\rho_{t_k}}(\beta_-^k)$ is small by the previous displayed estimate and $i(\gamma_i,\beta_-^k)$ is comparable to $q_{4k+1}$ by Lemma~\ref{lem:crossings} and because $\beta_-^k$ refers to the convergent prior to $\sigma_-^k$.

Thus the expression on the right is small compared to  $i(\gamma_i,\sigma_+^k) \ell_{\rho_{t_k}}(\beta_+^k)\asymp q_{4k+5}$.

Finally we consider the term coming from the collar $U$.
\begin{lem}\label{lem:collar} The term $\ell_{\rho_{t_k}}(\gamma_i\cap U)$ coming from lengths of arcs of $\gamma_i$ in the collar $U$ of the $(k+2)^{nd}$ slit is $o(q_{k+5})$
\end{lem}
 We have $$i(\gamma_i,\zeta^{k+2})\asymp q_{4k+1}.$$    
At time $t_k$  the horizontal component $h_{t_k}(\zeta^{k+2})\asymp \frac{1}{k^2}$  and the vertical component $v_{t_k}(\zeta^{k+2})\asymp \frac{q_{4k+1}}{q_{4k+5}}$. We can assume the former is larger as $a_k$ grows large quickly.  Since $|\zeta^{k+2}|^2_{t_k}\asymp 1/k^4$  and this a lower bound for extremal length, and extremal length is itself comparable to hyperbolic length, we find  $$\ell_{\rho_{t_k}}(\zeta^{k+2})\asymp \frac{1}{k^4},$$ and thus the diameter of the collar is $O(\log k)$. 
 
 Again the contribution to length of the twisting of each segment of $\gamma_i$ around the slit is $O(1)$.    
Thus $$\ell_{\rho_{t_k}}(\gamma_i\cap U)=O(q_{4k+1}\log k)=o(q_{4k+5})$$ for $a_{4k+3}$ large.

These estimates now finally say that the term  $i(\gamma_i,\sigma_+^k) \ell(\beta_+^k)$ dominate the other terms in Lemma~\ref{lem:asymptotic} so that the hypothesis of Lemma~\ref{lem:ergodiclimit} is satisfied.  We conclude  that we get the  ergodic endpoint
$[F,\mu_+]$ as limit point along times $t_k$. 

\begin{prop}\label{prop:barycenter for erg} Along the sequence
$$s_k=\frac{1}{2}\log(q_{4k+3}^2+p_{4k+3}^2)+k.$$
we have the barycenter as the limit in $\PMF$.
\end{prop}
 
Again  by the first conclusion of Proposition~\ref{prop:geom}  at time $s_k$ the surface $g_{s_k}r_\theta(Y,\omega_1)$ is within $O(\frac{1}{a_{4k+3}}+\frac{1}{a_{4k+4}})$  of two glued rectangular tori $T_-^k, T_+^k$ with vertical sides 1.   Then the  surface $g_{s_k+k}r_\theta(Y,\omega_1)$ is within $O(\frac{e^k}{a_{4k+3}}+\frac{e^k}{a_{4k+4}})$ of two glued rectangular tori $T_{\pm}^k$ with vertical sides $e^{-k}$. The length of the  horizontal side $\beta_+^k$ of 
$T_+^k$  is comparable to  $e^k$ and the length of the horizontal side of $T_-^k$ is at least $\frac 1 {4k^2}e^k$ by Lemma \ref{lem:erg}. 

  By Proposition~\ref{lem:largecylinder} on the corresponding limiting punctured tori $T_-^\infty,T_+^\infty$, with these flat lengths and hyperbolic metrics 
$\rho_-,\rho_+$, for some $D$, the ratio of the hyperbolic lengths of the horizontal curves $\beta_{\pm}$ satisfies \begin{equation}
\label{eq:closeto1}
1\leq \frac{\ell_{\rho_+}(\beta_+)}{\ell_{\rho_-}(\beta_-)}\leq \frac{4\log(e^k)+D}{4\log (\frac{e^k}{4k^2})-D}.
\end{equation}
We see that the right hand side goes to $1$ as $k\to \infty$.  
Now fix  $\epsilon>0$.  Then for large $k$,  the right hand side is at most $1+\epsilon/2$.  We would like to apply   Proposition~\ref{prop:augmented} and its consequence (\ref{eq:lengthsratio}). 

We have  punctured tori $T_-^k,T_+^k$ and almost horizontal crossing curves $\beta_-^k,\beta_+^k$.  We have, for $i=-,+$, the length $\ell_{\rho_{s_k+k}}(\gamma_i \cap (T_+^k\setminus U))\asymp q_{4k+5}\ell_{\rho_0}(\beta_+)$ and  $\ell_{\rho_{s_k+k}}(\gamma_i \cap (T_-^k\setminus U)) \asymp q_{4k+5}\ell_{\rho_0}(\beta_-)$. We need to check the following two conditions. 

  \begin{itemize}
\item  $g_{s_k+k}r_\theta(Y,\omega_1)$ is sufficiently close to  $T^\infty_-,T_+^\infty$.
\item the  hyperbolic length of $\gamma_i$ in the collar $U$ about the slit is small compared to the length of $\gamma_i$ in the complement of the collar in each torus $T_i^k$.  
\end{itemize}
 If these conditions hold then (\ref{eq:closeto1}) implies the desired  $|\frac{\ell_{\rho_{s_k+k}}(\gamma_1)}{\ell_{\rho_{s_k+k}}(\gamma_2)}-1|<\epsilon$. Since $\epsilon$ arbitrary we would be done.     

By choosing  $a_{4k+3},a_{4k+4}$  appropriately we can guarantee that  the first bullet is satisfied for $g_{s_k+k}r_\theta(Y,\omega_1)$.
 
The argument for the second bullet is similar to what has come before in the discussion of the length of slits $\zeta_{k+2}$.   
By Lemma~\ref{lem:slitsize}, we have 
$$h_{s_k+k}(\zeta_{k+2})\asymp \frac{e^{s_k+k}}{q_{4k+5}} \asymp  \frac{e^kq_{4k+3}}{k^2q_{4k+5}}$$
 and $$v_{s_k+k}(\zeta_{k+2})\asymp \frac{q_{4k+1}}{e^kq_{4k+3}}.$$  The twisting about the slit is $O(1)$ as in our discussion of the time $t_k$.  The number of intersections of $\gamma_i$ with the slit is comparable to $q_{4k+1}$. Thus the ratio of the contribution to 
$\ell_{\rho_{s_k+k}}(\gamma_i)$ from crossing the collar to length in the complement of collar in $T_+^k$ is   proportional to

\begin{equation}
\label{ratio2}
\frac{q_{4k+1}\max (-\log \frac{e^kq_{4k+3}}{k^2q_{4k+5}}, -\log \frac{q_{4k+1}}{e^kq_{4k+3}})}
{2\log (e^{2k})q_{4k+5}}.
\end{equation}
Here the term in the maximum comes from the length across the slit whose horizontal and vertical sizes were estimated just above, while it is multiplied by the intersection number of $\gamma_i$ with the slit, also as above; the denominator is the product of $i(\gamma_i, \sigma_+^k) \asymp q_{4k+5}$ by Lemma~\ref{lem:crossings}(1), and the length of $\gamma_i \cap (T^k_2\setminus U)$.  Since $\sigma_+^k$ has hyperbolic length comparable to its extremal length  which is comparable to  $e^{-2k}$, the arcs crossing the collar pick up lengths comparable to $-\log(e^{-2k})=2k$.

We can make this ratio as small as we want  by making $a_{4k+3}, a_{4k+4}$ large. 
similarly the length of $\gamma_i$ crossing the slit is small compared to the length crossing $T_-^k$.
The second bullet is satisfied.
This concludes the proof of (1).

\subsection{Proof of Theorem~\ref{thm:CF4}} Let $n_k=k$ so that $a_{n_k}=2^k$.  To establish the proof we wish to apply Lemma~\ref{lem:ergodiclimit}   at {\em all} times.   To do that we will again use Proposition~\ref{prop:augmented} and Lemma~\ref{lem:asymptotic}.  In particular, we will show that for each of our initial curves $\gamma_i$, the dominant term or terms in the formula in Lemma~\ref{lem:asymptotic} are the two terms that involve the length of $\gamma_i$ in the torus $T_+^k$: the contributions to the lengths from the torus $T_-^k$ will become increasingly negligible.  This will imply the hypothesis of Lemma~\ref{lem:ergodiclimit}, i.e. that in this case,
\begin{equation} \label{eqn:ergodiclimit}
\limsup_{s \to \infty}\frac{\ell_{\rho_{s}}(\gamma_i\cap T_-^k)}{\ell_{\rho_{s}}(\gamma_i \cap T_+^k)} = 0
\end{equation}

Set 
\begin{align*}
s_k &=\log q_{k-1}+\frac {k\log 2}{2} \\ 
t_k &=\log q_k.
\end{align*}
We will break up the interval $[s_k,s_{k+1}]$ into two intervals $[s_{k}, t_k]$  and $[t_k,s_{k+1}]$ and analyze each separately. 
\subsubsection{The subinterval $[t_k,s_{k+1}]$.}

We first consider the second interval $[t_k,s_{k+1}]$.  
As we have seen before,  by Lemma~\ref{lem:close to square} at time $t_k$ the larger torus $T_+^k$  has bounded geometry  with almost vertical curve $\sigma_+^k$ and almost horizontal curve $\beta_+^k$.   At the later time $s_{k+1}$, the torus $T_+^k$ is short in the vertical direction and long in the horizontal direction having basis vectors $\sigma_+^k$ and $\beta_+^k$  with lengths comparable to $(\frac{1}{\sqrt{a_{k+1}}},\frac 1 {\sqrt{a_{k+1}}})$ and $(\sqrt{a_{k+1}},o(1))$ respectively.

We focus first on the times $\{t_k\}$.
The change of area formula Lemma~\ref{lem:erg} (see also formula \eqref{tinytorus} and the discussion preceeding it) and the definition of $\beta_-^k$ implies that at time $t_k$, the surface $T_-^k$ is a torus with an approximately vertical side $\sigma_-^k$ with length comparable to $1$ and another side $\beta_-^k$ with vector comparable to $(2^{-k},2^{-k})$. 

\begin{lem}\label{lem:tk} Let $U$ be a collar about the slit as described in section~\ref{sec:hyperbolic}. Then $$\frac{\ell_{\rho_{t_k}}(\gamma_i\cap (T_-^k\setminus U))+\ell_{\rho_{t_k}}(\gamma_i\cap U)}{\ell_{\rho_{t_k}}(\gamma_i \cap T_+^k)} \asymp\frac{k}{2^k}.$$
\end{lem}
\begin{claim}$\ell_{\rho_{t_k}}(\gamma_i \cap (T_-^k\setminus U))\asymp\frac{kq_k}{2^k}.$
\end{claim}
\begin{proof}
By Lemma~\ref{lem:extremal},  $$\ell_{\rho_{t_k}}(\beta_-^k)\asymp 1/2^{k+1}.$$
Moreover since $\beta_i^k$ is the previous convergent, of length approximately $q_k/2^k$, we have
$$i(\gamma_i,\beta_-^k)\asymp q_k/2^k,$$
as once we fixed the curve $\gamma_i$ at the outset of the constructions, the intersection number above became comparable to the length of $\beta_-^k$.
 Thus, the first term appearing in Lemma~\ref{lem:asymptotic} 
\begin{equation}
\label{eq:first}
-i(\gamma_i,\beta_-^k)\log \ell_{\rho_{t_k}}(\beta_-^k)\asymp \frac{kq_k}{2^k}.
\end{equation}

The other term appearing in Lemma~\ref{lem:asymptotic} is 
\begin{equation}
\label{eq:second}
i(\gamma_i,\sigma_-^k)\ell_{\rho_{t_k}}(\beta_-^k)=O(q_k/2^{k+1}),
\end{equation}  which is much smaller than the  term  (\ref{eq:first}). This proves the claim.   
\end{proof}
\begin{proof}[Proof of Lemma \ref{lem:tk}]  We have, because $T_+^k$ has bounded geometry, that $$\ell_{\rho_{t_k}}(\gamma_i\cap (T_+^k\setminus U))\asymp  i(\gamma_i,\sigma_+^k)\ell_{\rho_{t_k}}(\beta_+^k)\asymp q_k.$$  By Lemma~\ref{lem:slitsize}, the flat length of the slit is comparable to $\frac 1 {2^k}$ and so its extremal length and therefore hyperbolic length is  at least a number comparable to $\frac{1}{2^{2k}}$. 

The slit is crossed  $\frac{q_k}{2^k}$
 times so that similarly to the claim, in the collar $U$, 
 $$\ell_{\rho_{t_k}}(\gamma_i\cap U)=O(\frac{kq_k}{2^k}).$$ 
 Thus, the contribution to $\ell_{\rho_t}(\gamma_i)$ is mostly from its intersection with $T_+^k\setminus U$ and not with $U$. We conclude that the denominator in the statement of the lemma is comparable to $q_k$, while the numerator, by the first claim, is at most a decaying factor $kq_k/2^k$ of that $q_k$.
The lemma follows.
\end{proof}

Having proved the limiting condition \eqref{eqn:ergodiclimit} for sequences of the form $\{t_k\}$, we extend the possible time parameters to  general times $u\in [t_k,s_{k+1}]$ in the interval $[t_k,s_{k+1}]$. We begin by estimating a majorant for the left-hand side of \eqref{eqn:ergodiclimit} in this regime of times.

\begin{lem} For $u\in [t_k,s_{k+1}]$  
 
 $$\frac{\ell_{\rho_u}(\gamma_i\cap (T_-^k\setminus U))+\ell_{\rho_u}(\gamma_i\cap U))}{\ell_{\rho_u}(\gamma_i \cap T_+^k)}= O(\frac{1}{k}).$$ 
\end{lem}
\begin{proof}
The denominator in the above expression is  comparable to   $$-i(\gamma_i,\sigma_+^k)\log\ell_{\rho_u}(\beta_+^k)\asymp q_k(1+\log(e^{2(u-t_k)}))\asymp q_k(1+2(u-t_k)).$$
We now direct our attention to the numerator.
Consider the term  $-\log \ell_{\rho_u}(\beta_-^k)i(\gamma_i,\beta_-^k)$ which is a summand for the first term in the numerator: See Proposition~\ref{prop:augmented}. As $u$ increases, since $\ell_{\rho_u}(\beta_-^k)$ increases, this product term decreases from its value at $t_k$. In particular, it is bounded by this original value of $O(kq_k/2^k)$.  

By Lemma~\ref{lem:extremal},   the other summand in the first term in the numerator is
$$i(\gamma_i,\sigma_-^k)\ell_{\rho_u}(\beta_-^k)=O(q_k e^{2(u-t_k)}/2^k).$$
Therefore the first term in the numerator  
  $$\ell_{\rho_u}(\gamma_i\cap (T_-^k\setminus U))=O(\max (\frac{kq_k}{2^k},\frac{e^{2(u-t_k)}q_k}{2^k})).$$ The hyperbolic length of the part of $\gamma_i$ crossing the slit remains small: at time $t_k$, the estimates of Lemma~\ref{lem:tk} show that this length is of order $\frac{kq_k}{2^k}$, but then as $u$ increases from zero, the size of the slit also grows, forcing the diameter of the collar to shrink, driving down the length of the portion of the curve $\gamma_i$ crossing the collar.
Now $0\leq u-t_k\leq  (\log 2) (k+1)/2$.
Since $\frac{e^x}{1+x}$ is increasing, the ratio appearing in the statement of the lemma increases as $u-t_k$ increases.  At time $u-t_k=(\log 2) (k+1)/2$, that ratio is  $O(1/k)$.
The lemma follows.
\end{proof}

This lemma above proves the needed estimate \eqref{eqn:ergodiclimit} for the sequences drawn from intervals $[t_k,s_{k+1}]$. Next we consider the remaining interval $[s_k, t_k]$.

\subsubsection{The interval $[s_k, t_k]$.} We begin by considering separately the case of times $\{s_k\}$, following a different analysis from that in the previous subsection.


\begin{lem} \label{lem:sk}
$$\frac{\ell_{\rho_{s_k}}(\gamma_i\cap (T_-^k\setminus U))+\ell_{\rho_{s_k}}(\gamma_i\cap U)}{\ell_{\rho_{s_k}}(\gamma_i \cap T_+^k)} \asymp\frac 1 k.$$
\end{lem}
The lemma will immediately follow from the next two claims, in which we separately estimate the denominator and numerator of the fraction in the lemma.
\begin{claim}${\ell_{\rho_{s_k}}(\gamma_i\cap T_+^k)}\asymp\frac{kq_k}{2^k} $
\end{claim}
\begin{proof}
$T_+^k$ has an almost vertical curve $\sigma_+^k$ 
with $|\sigma_+^k|_{s_k}=a_k^{1/2}=2^{k/2}$ and a short curve $\beta_+^k$  with horizontal and vertical components comparable to  $(2^{-\frac k 2},2^{-\frac k 2})$. By Lemma~\ref{lem:extremal}, the hyperbolic length $\ell_{\rho_{s_k}}(\beta_+^k)\asymp 1/a_k=2^{-k}$. 
We have again ${i(\gamma_i, \sigma_+^k)\asymp  q_{n_k}}$ and so $$i(\gamma_i,\sigma_+^k)\ell_{\rho_{s_k}}(\beta_+^k)\asymp  
q_{n_k}/2^k.$$ On the other hand $$-2i(\gamma_i,\beta_+^k)\log \ell_{\rho_{s_k}}(\beta_+^k)\asymp q_{n_k}/2^k\log 2^k\asymp q_{n_k}k/2^k,$$ and this is the dominant term in the expansion in Lemma~\ref{lem:asymptotic}. 
\end{proof}
\begin{claim}$\ell_{\rho_{s_k}}(\gamma_i\cap (T_-^k\setminus U))\asymp\frac{q_k}{2^k}.$
\end{claim}
\begin{proof}

We begin by estimating the lengths of the basis set $\beta_-^k$ and $\sigma_-^k$.  At time $t_k$, the curve $\sigma_-^k$ is almost vertical with vertical length comparable to one and horizontal length comparable to $2^{-k}$; at the same time $t_k$, the curve $\beta_-^k$ has vertical length comparable to $2^{-k}$, and horizontal length -- estimated by the area formula \eqref{tinytorus} -- comparable to $2^{-k}$.  We then flow backwards to the time $s_k$, and find that 
the  short curve $\beta_-^k$ now has vertical component comparable to $2^{-k/2}$ and horizontal component comparable to $2^{-\frac{3k}{2}}$. Similarly, the long curve $\sigma_-^k$ is almost vertical with horizontal component comparable to $1/2^{k/2}$ and vertical component  comparable to $2^{k/2}$.  The flat area of this torus is still roughly $2^{-k}$.   If we  were to  normalize $T_-^k$ so it has unit area then $\sigma_-^k$ has length compable to $2^k$ and $\beta_-^k$ has vertical component of length comparable to $1$ and horizontal component comparable to $1/2^k$. 
Because the basis element $\beta_-^k$ has vertical component of length comparable to one, we can choose a different basis for the torus with proportional horizontal and vertical lengths: in particular, we see that  $T_-^k$  lies in a compact set of tori.   This means that -- without renormalizing areas -- there is a basis consisting of an almost vertical curve $\beta_-^k$ and some other curve $\alpha_-^k$ of comparable length which is not almost vertical.  By  Lemma~\ref{lem:extremal} the extremal lengths, hence hyperbolic lengths, of the basis curves are comparable to $1$.  Because we can compute intersection numbers on the original square torus, we have $i(\gamma_i,\beta_-^k)\asymp \frac{q_k}{2^k}$ and  since $\gamma_i$ is almost horizontal, 
 we also have $i(\gamma_i,\alpha^k_-)=O(\frac{q_k}{2^{k}})$.      The claimed estimate follows.
\end{proof}
 
We also note that 
it follows from Lemma~\ref{lem:crossings} that at time $s_k$ both the slit $\zeta^k$ and previous slit $\zeta^{k-1}$ have flat length comparable to $\frac{1}{2^{k/2}}$.  

We next subdivide the interval $[s_k, t_k]$ into two subintervals: the idea is that different terms will dominate in the expansion of the fraction in \eqref{eqn:ergodiclimit} in the two different subintervals.

 Let $z_k$ satisfy $$e^{2z_k}=k\log 2-2z_k$$ and divide the time interval $[s_k,t_k]$ into $[s_k,s_k+z_k]$ and $[s_k+z_k,t_k]$. 

Notice \begin{equation}
\label{eq:z_k}
z_k \asymp \log k.
\end{equation}

\begin{lem}
On $[s_k,s_k+z_k]$, we have the estimate
$$\frac{\rho_{s_k+t}(\gamma_i\cap (T_-^k\setminus U))+\rho_{s_k+t}(\gamma_i\cap U)}{\rho_{s_k+t}(\gamma_i \cap T_+^k)}\asymp \frac{t}{k\log 2 -2t}=O(\frac{\log k}{k}).$$
\end{lem}

\begin{proof}

We first consider the denominator.
We have by Lemma~\ref{lem:extremal} that \begin{equation}
\label{eq:dominant2}
i(\gamma_i,\sigma_+^k)\ell_{\rho_{s_k+t}}(\beta_+^k)\asymp \frac{q_ke^{2t}}{2^k}.\end{equation}
The other term in Lemma~\ref{lem:asymptotic} \begin{equation}
\label{eq:dominant1}
-i(\gamma_i,\beta_+^k)\log \ell_{\rho_{s_k+t}}(\beta_+^k)\asymp   
 -\frac{q_k}{2^k}\log (e^{2t}/2^k)= (k\log 2-2t)q_k/2^k.
\end{equation}
If $t\leq z_k$ then (\ref{eq:dominant1}) dominates (\ref{eq:dominant2}). We consider next the numerator.
On $T_-^k$, at time $s_k+t$,  for any time $t$,  the short curve $\beta_-^k$ after normalization has flat vertical length $e^{-t}$, hence hyperbolic length comparable to $e^{-2t}$ so $$-i(\gamma_i,\beta_-^k)\log \ell_{\rho_{s_k+t}}(\beta_-^k)\asymp tq_k/2^k.$$ The other term we need to consider, $$i(\gamma_i, \sigma_-^k)\ell_{\rho_{s_k+t}}(\beta_-^k) \asymp \frac{q_k}{2^k}e^{-2t}.$$ is smaller for all times $t$.

As for the hyperbolic length of the slit $\zeta^k$, we see that 
at time $s_k+t$ the flat length is comparable to $e^{-t}/2^{k/2}$ while the flat length of $\zeta^{k-1}$ is $e^t/2^{k/2}$. The second consequence of Lemma~\ref{lem:extremal}  says that the extremal hence hyperbolic length of $\zeta^k$ is comparable to $\frac{1}{2t}$ which is smaller than $e^{-2t}$.  
Thus the second term in the numerator is dominated by the first. 

The first bound in the  Lemma then follows from these estimates.  The second   bound follows from (\ref{eq:z_k}).

\end{proof}

\begin{cor}\label{cor:decrease}The hyperbolic length of $\gamma$ at $s_k$ is larger  than at $s_k+z_k$.
\end{cor}
\begin{proof}  
This is true for the quantity in (\ref{eq:dominant1}) which dominates the other terms.
\end{proof}
Finally, we consider the remaining subinterval $[s_k+z_k,t_k]$ of $[s_k,t_k]$, which is the final interval for which we have not evaluated the expression in \eqref{eqn:ergodiclimit}.

\begin{lem}
For  $t \in [s_k+z_k,t_k]$, we have 
$$\frac{\rho_{s_k+t}(\gamma_i\cap (T_-^k\setminus U))+\rho_{s_k+t}(\gamma_i\cap U))}{\rho_{s_k+t}(\gamma_i \cap T_+^k)}= O(\frac{t}{e^{2t}})=O(\frac{\log k}{k})$$
\end{lem}
\begin{proof}
For these times   the denominator is comparable to 
(\ref{eq:dominant2}).  Since the numerator is  $O(tq_k/2^k)$,  the first bound follows. The second bound again follows from  (\ref{eq:z_k}). 
\end{proof}

We now finish the proof of  Theorem~\ref{thm:CF4}.:  We have proven the required limit  \eqref{eqn:ergodiclimit} for times in the  $[t_k,s_{k+1}]$ in Lemma~\ref{lem:sk}.   In  the interval  $[s_k,t_k]$ in this second subsection the required estimate  follows from the last two lemmas.

\end{document}